\documentclass[reqno,a4paper]{amsart}
\usepackage[utf8]{inputenc}
\usepackage{amsmath,amssymb}
\usepackage{mathrsfs}
\usepackage{yhmath}
\usepackage{booktabs,url}

\usepackage[font=footnotesize,labelfont=bf]{caption}

\newcommand{\de}[0]{\mathrel{\mathop:}=}
\newcommand{\ie}[0]{\mathrm{i}}
\newcommand{\dif}[1]{\mathrm{d}#1}
\newcommand{\C}[0]{\mathbb{C}}

\newcommand{\N}[0]{\mathbb{N}}

\usepackage{amsmath}

\newtheorem{theorem}{Theorem}

\newtheorem{corollary}{Corollary}
\newtheorem{lemma}{Lemma}

\theoremstyle{definition}
\newtheorem{remark}{Remark}

\title[Explicit estimates for $\zeta(s)$ in the critical strip under RH]{Explicit estimates for $\zeta(s)$ in the critical strip under the Riemann Hypothesis}
\author{Aleksander Simoni\v{c}}
\address{School of Science, The University of New South Wales (Canberra), ACT, Australia}
\email{a.simonic@student.adfa.edu.au}
\date{\today}

\begin{document}

\begin{abstract}
Assuming the Riemann Hypothesis, we provide effective upper and lower estimates for $\left|\zeta(s)\right|$ right to the critical line. As an application we make explicit Titchmarsh's conditional bound for the Mertens function and Montgomery--Vaughan's conditional bound for the number of $k$-free numbers.
\end{abstract}

\maketitle
\thispagestyle{empty}

\section{Introduction}

Let $\zeta(s)$ be the Riemann zeta-function and $s=\sigma+\ie t$, where $\sigma$ and $t$ are real numbers. One of the great problems in zeta-function theory is to determine the true order of $|\zeta(s)|$ in the critical strip. Due to the functional equation we can assume $1/2 \leq\sigma\leq 1$. Having $\zeta(s)\ll_{\varepsilon} |t|^{\mu(\sigma)+\varepsilon}$ for $\varepsilon>0$ where
\[
\mu(\sigma) \de \limsup_{t\to\infty}\frac{\log{\left|\zeta\left(\sigma+\ie t\right)\right|}}{\log{t}},
\]
it is well known that $0\leq\mu(\sigma)\leq (1-\sigma)/2$. This estimate can be improved for certain values $\sigma$, see~\cite[Chapter 5]{Titchmarsh} and also~\cite{TrudgianANewUpper,TrudgianExplLogDer,Platt,HiaryAnExplicit,BourgainDecoupling,Patel,PatelPhD} to mention only recent results. Note that the Lindel\"{o}f Hypothesis states that $\mu(\sigma)\equiv0$. On the other hand, the Riemann Hypothesis (RH) drastically improves these estimates since it implies that
\begin{equation}
\label{eq:someEstimateForZeta}
\log{\zeta(s)} \ll_{\varepsilon,\sigma_0} \left(\log{t}\right)^{2(1-\sigma)+\varepsilon}
\end{equation}
for $\varepsilon>0$, $1/2<\sigma_0\leq\sigma\leq 1$ and $t$ large, see~\cite[Theorem 14.2]{Titchmarsh}. In particular, the bound~\eqref{eq:someEstimateForZeta} implies Lindel\"{o}f-type estimates for $\zeta(s)$ and $1/\zeta(s)$, see~\eqref{eq:someEstimateForRecZeta}. Titchmarsh~\cite{TitchConseq} provided quantitative versions of~\eqref{eq:someEstimateForZeta}, namely that RH guarantees
\begin{gather}
\log{\left|\frac{1}{\zeta(s)}\right|} \ll \frac{\log{t}}{\log\log{t}}\log{\frac{2}{\left(\sigma-\frac{1}{2}\right)\log{\log{t}}}}, \label{eq:TitchZetaBound} \\
\log{\left|\zeta(s)\right|} \ll \frac{\log{t}}{\log\log{t}} \label{eq:TitchZetaBound2}
\end{gather}
for $0<\sigma-1/2\ll 1/\log{\log{t}}$ and $t$ large. Estimates~\eqref{eq:TitchZetaBound} and~\eqref{eq:TitchZetaBound2} are of main interest near the critical line since
\begin{equation}
\label{eq:someEstimateForZeta2}
\log{\zeta(s)} \ll \frac{\left(\log{t}\right)^{2(1-\sigma)}}{\log{\log{t}}}
\end{equation}
for $1/\log{\log{t}}\ll\sigma-1/2\leq 1/2-\varepsilon$, $\varepsilon>0$ and $t$ large, see~\cite[Equation 14.14.5]{Titchmarsh}. It should also be noted that in the case $\sigma=1$ one can replace the right-hand side of~\eqref{eq:TitchZetaBound} and~\eqref{eq:TitchZetaBound2} with $\log{\log{\log{t}}}$, see~\cite[Theorem 14.9]{Titchmarsh}, and~\cite{Lamzouri,LanguascoTrudgian} for explicit results of similar kind.

The main purpose of this paper is to provide conditional effective versions of~\eqref{eq:TitchZetaBound} and~\eqref{eq:TitchZetaBound2}, see the following theorem.

\begin{theorem}
\label{thm:main}
Assume the Riemann Hypothesis. For $1/2<\sigma\leq 3/2$ and $2\exp{\left(e^2\right)}\leq t\leq T$ we have
\begin{gather}
\log{\left|\frac{1}{\zeta(s)}\right|} \leq \omega_1\left(\sigma,t_0;\log{\log{\frac{3T}{2}}}\right)\log{\frac{3T}{2}}, \label{eq:corMain} \\
\log{\left|\zeta(s)\right|} \leq \omega_2\left(t_0;\log{\log{\frac{3T}{2}}}\right)\log{\frac{3T}{2}}. \label{eq:corMain2}
\end{gather}
Here $2\gamma_1\leq t_0\leq 50$ with $\gamma_1$ being the positive ordinate of the first nontrivial zero, and the functions $\omega_1$ and $\omega_2$, which, as $T\to\infty$ are asymptotic to $1/\log{\log{(3T/2)}}$, see~\eqref{eq:limits} and~\eqref{eq:limits2}, are given by~\eqref{eq:omega1} and~\eqref{eq:omega2}, respectively. In particular,
\begin{gather}
\log{\left|\frac{1}{\zeta(s)}\right|} \leq \frac{1.756\log{(2T)}}{\log{\log{(2T)}}}\log{\left(1+\frac{37.345}{\left(\sigma-\frac{1}{2}\right)\log{\log{T}}}\right)} + \frac{2.51\log{(2T)}}{\log{\log{(2T)}}}, \label{eq:corMainConrete} \\
\log{\left|\zeta(s)\right|} \leq \frac{8.45\log{(2T)}}{\log{\log{(2T)}}} \label{eq:corMain2Conrete}
\end{gather}
for $T\geq (2/3)\exp{\left(e^{10}\right)}$.
\end{theorem}

Theorem~\ref{thm:main} is included in Corollary~\ref{cor:main}, the proof of which will be provided in Section~\ref{sec:MainBound}. Observe that~\eqref{eq:TitchZetaBound2} implies
\[
\left|\zeta\left(\frac{1}{2}+\ie t\right)\right| \leq \exp{\left(C\frac{\log{t}}{\log\log{t}}\right)}
\]
for some $C>0$ and $t$ large, and~\eqref{eq:corMain2} or~\eqref{eq:corMain2Conrete} make this effective. This was recently made explicit by the author using different techniques, and the present work can be viewed as an application of the results from~\cite{SimonicSonRH}; we are using Titchmarsh's approach~\cite{TitchConseq} which depends on conditional bounds for $S(t)$ and $S_1(t)$, see Section~\ref{sec:MainBound}. However, inequality~\eqref{eq:corMain2Conrete} provides a worse estimate for $\left|\zeta\left(1/2+\ie t\right)\right|$ than~\cite[Corollary 1]{SimonicSonRH}.

In addition to Theorem~\ref{thm:main}, we also give two applications of such bounds, namely effective conditional estimates for the Mertens function $M(x)=\sum_{n\leq x}\mu(n)$ and for the number of $k$-free numbers $Q_k(x)$. As usual, $x\geq1$, $k\geq 2$ is an integer, and $\mu(n)$ is the M\"{o}bius function. Our main result in this direction is summarized by the next theorem.

\begin{theorem}
\label{thm:FinalApp}
Assume the Riemann Hypothesis. Then
\begin{equation}
\label{eq:thmMertensGen}
\left|M(x)\right| \leq 0.6\sqrt{x}\exp{\left(\frac{5.2251\log{x}}{\log{\log{\left(\frac{3e}{2}\sqrt{x}\right)}}}+\log{\log{x}}\right)}
\quad \textrm{for} \quad x\geq 10^{10^{4.487}}
\end{equation}
and
\begin{multline}
\label{eq:thmkFree}
\left|Q_{k}(x)-\frac{x}{\zeta(k)}\right| \leq 0.1x^{\frac{1}{k+1}}\exp\left(\frac{7.525\log{x}}{2(k+1)\log{\log{\left(\frac{3e}{2}x^{\frac{1}{2(k+1)}}\right)}}}\right. \\
\left.+\frac{\left(7.492k+8.25\right)\log{x}}{(k+1)\log{\log{\left(\frac{3e}{2}x^{\frac{1}{k+1}}\right)}}}+2\log{\log{x}}\right) \quad \textrm{for} \quad x\geq 10^{(k+1)10^{23.147}}.
\end{multline}
Additionally,
\begin{equation}
\label{eq:thmMertens}
\left|M(x)\right| \leq Ax^{\alpha}\log{x} \quad \textrm{for} \quad x\geq 10^{10^{X}}
\end{equation}
and
\begin{equation}
\label{eq:thmSquarefree}
\left|Q_2(x)-\frac{x}{\zeta(2)}\right| \leq Bx^{\beta}\log^{2}{x} \quad \textrm{for} \quad x\geq 10^{3\cdot10^{Y}},
\end{equation}
where values for the constants are given by Table~\ref{tab:FinalApp}.
\end{theorem}

\begin{table}
   \centering
\begin{footnotesize}
\begin{tabular}{cllllllllll}
\toprule
$\alpha$ & $0.999$ & $0.99$ & $0.9$ & $0.85$ & $0.8$ & $0.75$ & $0.7$ & $0.65$ & $0.6$ & $0.55$ \\
$X$ & $4.487$ & $4.543$ & $5.255$ & $5.837$ & $6.654$ & $7.865$ & $9.798$ & $13.003$ & $19.53$ & $39.121$ \\
$A$ & $0.517$ & $0.505$ & $0.408$ & $0.362$ & $0.322$ & $0.286$ & $0.254$ & $0.227$ & $0.202$ & $0.179$ \\
\midrule
$\beta$ & $0.4999$ & $0.499$ & $0.49$ & $0.48$ & $0.47$ & $0.46$ & $0.45$ & $0.44$ & $0.4$ & $0.35$ \\
$Y$ & $23.147$ & $23.273$ & $24.628$ & $26.326$ & $28.276$ & $30.534$ & $33.179$ & $36.314$ & $58.29$ & $234.21$ \\
$B$ & $0.0852$ & $0.0850$ & $0.0826$ & $0.0799$ & $0.0774$ & $0.0750$ & $0.0726$ & $0.0704$ & $0.0622$ & $0.0533$ \\
\bottomrule
\end{tabular}
\end{footnotesize}
   \caption{Values for the constants from Theorem~\ref{thm:FinalApp}.}
   \label{tab:FinalApp}
\end{table}

One function which is commonly associated with $M(x)$ is $m(x)=\sum_{n\leq x}\mu(n)/n$. It is well known that $m(x)=o(x)$. The following corollary to inequality~\eqref{eq:thmMertens} can be viewed as the conditional and effective version of this relation. Later we will give a quantitative formulation of the assertion~\cite[Theorem 14.25 (A)]{Titchmarsh} that RH implies that $\sum_{n=1}^{\infty}\mu(n)n^{-s}$ converges to $1/\zeta(s)$ for every $\sigma>1/2$, see Theorem~\ref{thm:m}.

\begin{corollary}
\label{cor:m}
Assume the Riemann Hypothesis. Then
\begin{equation}
\label{eq:m}
\left|m(x)\right| \leq \frac{51.005\log{x}+5050}{x^{0.01}}
\end{equation}
for $x\geq 10^{10^{4.543}}$.
\end{corollary}

Our proof of Theorem~\ref{thm:FinalApp} strongly relies on the ideas from Titchmarsh~\cite{TitchConseq}, and from Montgomery and Vaughan~\cite{MontyVaughanSquarefree}. It should be mentioned here that~\eqref{eq:thmMertens} and~\eqref{eq:thmSquarefree} do not simply follow from~\eqref{eq:thmMertensGen} and~\eqref{eq:thmkFree}, and special consideration is needed to provide values for $X$, $A$, and for $Y$, $B$ for selected $\alpha$ and $\beta$, respectively. Moreover, we identify some irregularities in Titchmarsh's later proof~\cite[Chapter 14]{Titchmarsh} and in Montgomery--Vaughan's proof, see Sections~\ref{sec:kfree} and~\ref{sec:MainBound}.

While comparing Theorem~\ref{thm:FinalApp} with other (unconditional) explicit results, it is not hard to verify that even constants from the first column in Table~\ref{tab:FinalApp} will produce better bounds than~\eqref{eq:Ramare},~\eqref{eq:Chalker}, and~\eqref{eq:Cohen}. The same is true also for~\eqref{eq:m}. This is mainly because of our large admissible values for $x$, which is a consequence of the double logarithm from Theorem~\ref{thm:main} and also of relatively large constants in estimates for $S(t)$ and $S_1(t)$, see~\eqref{eq:SExpl} and~\eqref{eq:S1Expl}. Possible improvement upon~\eqref{eq:thmMertens} by using~\eqref{eq:someEstimateForZeta} or~\eqref{eq:someEstimateForZeta2} is briefly described in Remark~\ref{rem:RHB}.

The outline of this paper is as follows. In Section~\ref{sec:app} we list some results related to Theorem~\ref{thm:FinalApp} and sketch the original proofs while Section~\ref{sec:MainBound} is devoted to the proof of Theorem~\ref{thm:main}. General bounds for $M(x)$ and $Q_k(x)$ are formulated and proved in Section~\ref{sec:Perron}, where also an effective truncated Perron's summation formula (Theorem~\ref{thm:truncatedPerron}) is given. The proofs of Theorem~\ref{thm:FinalApp} and Corollary~\ref{cor:m} are provided in Section~\ref{sec:proofApp}.

\section{Two applications}
\label{sec:app}

The aim of this section is to discuss some results and relevant techniques concerning upper bounds for partial sums of the M\"{o}bius function $M(x)$ and for the number of $k$-free numbers $Q_k(x)$.

\subsection{On the Mertens function $M(x)$}

It is well known that the assertion $M(x)=o(x)$ is equivalent to the Prime Number Theorem. The most recent explicit version of this relation is provided by Ramar\'{e}~\cite{Ramare2013}, namely
\begin{equation}
\label{eq:Ramare}
\left|M(x)\right| \leq \frac{0.013\log{x}-0.118}{\log^{2}{x}}x
\end{equation}
is true for $x\geq 1.08\cdot10^6$. Applying the zero-free regions for $\zeta(s)$ will improve bounds for $M(x)$, e.g., the Vinogradov--Korobov region~\cite[Chapter 6]{Ivic} implies the strongest unconditional estimate
\[
M(x) \ll x\exp{\left(-C\left(\log{x}\right)^{\frac{3}{5}}\left(\log{\log{x}}\right)^{-\frac{1}{5}}\right)}
\]
for some absolute constant $C>0$, see~\cite[p.~191]{WalfiszWeyl}. Effective results which are based on the classical zero-free region and are thus in the form
\begin{equation}
\label{eq:Chalker}
\left|M(x)\right| \leq C_{W}x\exp{\left(-c_{W,\varepsilon}\sqrt{\log{x}}\right)}, \quad x > x_0\left(W,\varepsilon\right)
\end{equation}
for some known positive constants $C_{W}$, $c_{W,\varepsilon}$ and $x_0\left(W,\varepsilon\right)$, e.g., $C_{W}=6.1\cdot10^{8}$, $c_{W,\varepsilon}=0.297$ and $x_0\left(W,\varepsilon\right)=\exp{\left(14305.32\right)}$, are given in~\cite{Chalker}. Both~\cite{Ramare2013} and~\cite{Chalker} provide also explicit bounds for $m(x)$ which are based on~\eqref{eq:Ramare} and~\eqref{eq:Chalker}.

Littewood~\cite{LittlewoodMobius} proved that RH is equivalent to
\begin{equation}
\label{eq:mobiusDef}
M(x) \ll_{\varepsilon}x^{\frac{1}{2}+\varepsilon}
\end{equation}
for every $\varepsilon>0$. The proof can be divided into three groups:
\begin{enumerate}
\item Partial summation gives
\begin{equation}
\label{eq:PartSumMobius}
\frac{1}{s\zeta(s)} = \int_{0}^{\infty} \frac{M(x)}{x^{s+1}}\dif{x}
\end{equation}
for $\sigma>1$.
\item Truncated Perron's formula implies
\begin{equation}
\label{eq:PerronNonexplicit}
M(x) = \frac{1}{2\pi\ie}\int_{c-\ie T}^{c+\ie T}\frac{x^{s}\dif{s}}{s\zeta(s)} + O\left(\frac{x^{c}}{(c-1)T}\right) + O\left(\frac{x\log{x}}{T}\right) + O(1)
\end{equation}
for $c>1$ and $T\geq 1$.
\item RH guarantees
\begin{equation}
\label{eq:someEstimateForRecZeta}
\frac{1}{\zeta(s)} \ll_{\varepsilon,\sigma_0} |t|^{\varepsilon}
\end{equation}
for $1/2<\sigma_0\leq\sigma\leq 1$.
\end{enumerate}
If~\eqref{eq:mobiusDef} is true, then the integral in~\eqref{eq:PartSumMobius} is a holomorphic function on the half-plane $\sigma>1/2+\varepsilon$ and thus so too is $1/\zeta(s)$ by analytic continuation. Conversely, RH allows us to move the line of integration in~\eqref{eq:PerronNonexplicit} arbitrarily close to the critical line, while~\eqref{eq:someEstimateForRecZeta} and the choice $T=x^c$ for $c=1+1/\log{x}$ then imply~\eqref{eq:mobiusDef}, see~\cite[Theorem 14.25 (C)]{Titchmarsh} for details.

Since Littlewood's paper there has been made some progress to determine $\varepsilon$ from~\eqref{eq:mobiusDef} as a function of $x$. Landau~\cite{LandauMobius} showed that $\varepsilon \ll \log{\log{\log{x}}}/\log{\log{x}}$, while Titchmarsh~\cite{TitchConseq} improved this to
\begin{equation}
\label{eq:TitchMobius}
M(x) \ll \sqrt{x}\exp{\left(\frac{C\log{x}}{\log{\log{x}}}\right)}
\end{equation}
for some $C>0$ by using~\eqref{eq:TitchZetaBound}. For a long time Titchmarsh's result has not been improved until Maier and Montgomery~\cite{MaierMontgomery} proved that $\varepsilon\ll \left(\log{x}\right)^{-22/61}$. Shortly after their proof was published, Soundararajan~\cite{SoundMobius} obtained $\varepsilon\ll\left(\log{\log{x}}\right)^{14}\left(\log{x}\right)^{-1/2}$. According to~\cite{Balazard}, it is possible to refine his proof and show that
\[
M(x) \ll_{\varepsilon} \sqrt{x}\exp{\left(C\sqrt{\log{x}}\left(\log{\log{x}}\right)^{\frac{5}{2}+\varepsilon}\right)},
\]
which is currently the strongest known result while assuming only RH. For better estimates under the assumption of other plausible conjectures surrounding the zeta-function theory, see~\cite{Ng2004,Saha}. For a generalisation of Soundararajan's estimate to partial sums of $\mu(n)$ in arithmetic progressions, see~\cite{HalupczokSuger}.


Having Theorem~\ref{thm:main} at disposal, it is not hard to obtain an effective version of~\eqref{eq:TitchMobius} by following procedure with steps (2) and (3) from Littlewood's proof. We choose $c=1+1/\log{x}$ and $T=e\sqrt{x}$, see Corollary~\ref{cor:Perron}. Then we move the line of integration to some $\sigma_0\in(1/2,1)$ and derive bounds for the resulting integrals by using~\eqref{eq:corMain} and Lemma~\ref{lem:zetabound}, see the proof of Theorem~\ref{thm:MainMoebius} for all details. Estimate~\eqref{eq:thmMertensGen} will follow after taking $\sigma_0-1/2\ll 1/\log{\log{x}}$ for suitable chosen constants, while the proof of~\eqref{eq:thmMertens} now consists of optimizing $\sigma_0$ and $t_0$ in order to obtain desired values for $\alpha$.

\begin{remark}
\label{rem:RHB}
It might be possible to improve $X$ from~\eqref{eq:thmMertens} for $\alpha$ close to $1$ by using effective versions of~\eqref{eq:someEstimateForZeta} or~\eqref{eq:someEstimateForZeta2} in combination with
\begin{equation}
\label{eq:Mhat}
\widehat{M}(x+h) - \widehat{M}(x) = \frac{1}{2\pi\ie}\int_{\sigma_0-\ie\infty}^{\sigma_0+\ie\infty} \frac{(x+h)^{s+1}-x^{s+1}}{s(s+1)\zeta(s)}\dif{s},
\end{equation}
where $\widehat{M}(x)\de\sum_{n\leq x}(x-n)\mu(n)$, $0<h\leq x$ and $\sigma_0\in\left(1/2,1\right)$. Formula~\eqref{eq:Mhat} follows after moving the line of integration in the corresponding Perron's formula, the process which is justified under the assumption of RH. Observe that
\begin{equation}
\label{eq:Mhat2}
\left|\left(\widehat{M}(x+h) - \widehat{M}(x)\right)h^{-1} - M(x)\right| \leq h+1.
\end{equation}
Take $\sigma_0\in\left(1/2,1\right)$. By~\eqref{eq:someEstimateForZeta2} there exist $t_0>0$ and $C>0$ such that
\[
\left|\frac{1}{\zeta\left(\sigma_0+\ie t\right)}\right| \leq t^{\varepsilon_0}, \quad
\varepsilon_0 = \frac{C}{\left(\log{t_0}\right)^{2\sigma_0-1}\log{\log{t_0}}} < 1
\]
for $t\geq t_0$. Take $h=x^{\kappa}$ for $0<\kappa\leq 1$. By using the reflection principle, splitting the range of integration in~\eqref{eq:Mhat} into three parts from $\sigma_0$ to $\sigma_0+\ie t_0$ to $\sigma_0+\ie xh^{-1}$ to $\sigma_0+\ie \infty$, and then using~\eqref{eq:Mhat2} together with
\[
\left|\frac{(x+h)^{s+1}-x^{s+1}}{s+1}\right| \leq h\left(x+h\right)^{\sigma_0} \ll hx^{\sigma_0}
\]
for the first two integrals and $\left|(x+h)^{s+1}-x^{s+1}\right|\leq 2\left(x+h\right)^{\sigma_0+1}\ll x^{\sigma_0+1}$ for the last, we obtain
\[
M(x) \ll h + x^{\sigma_0} + \left(\frac{x}{h}\right)^{\varepsilon_0} x^{\sigma_0} \ll x^{\kappa} + x^{\sigma_0+\left(1-\kappa\right)\varepsilon_0}.
\]
The problem is now to optimize $\kappa$. However, we must emphasize that the implied constants in the latter inequality depend also on knowing the precise behaviour of $\left|\zeta\left(\sigma_0+\ie t\right)\right|$ for all $t\in\left[0,t_0\right]$ where $t_0$ could be large.
\end{remark}

\subsection{On the number of $k$-free numbers}
\label{sec:kfree}

Let $k\geq 2$ be an integer. We say that $n\in\N$ is a $k$-free number if it has no nontrivial divisor which is a perfect $k$th power. Denote by $Q_k(x)$ the number of $k$-free numbers not exceeding $x\geq 1$. It is not hard to prove by elementary methods that
\begin{equation}
\label{eq:kfree}
Q_k(x) = \sum_{n\leq x}\sum_{d^{k}|n}\mu(d) = \frac{x}{\zeta(k)} + O\left(x^{\frac{1}{k}}\right),
\end{equation}
where implicit constants depend on $k$, see~\cite[Theorem 2.2]{MontgomeryVaughan} in the case $k=2$. Cohen et al.~\cite{CohenDressMarraki} provide an explicit bound
\begin{equation}
\label{eq:Cohen}
\left|Q_2(x)-\frac{x}{\zeta(2)}\right| \leq 0.02767\sqrt{x}
\end{equation}
for $x\geq 4.4\cdot10^5$, while the author is not aware of any published effective estimates on $Q_k(x)$ for larger $k$. Similarly as before, the Vinogradov--Korobov zero-free region implies the strongest unconditional estimate
\[
Q_k(x) = \frac{x}{\zeta(k)} + O\left(x^{\frac{1}{k}}\exp{\left(-Ck^{-\frac{8}{5}}\left(\log{x}\right)^{\frac{3}{5}}\left(\log{\log{x}}\right)^{-\frac{1}{5}}\right)}\right)
\]
for $C>0$ and the implied constants may depend on $k$, see~\cite[p.~192]{WalfiszWeyl}. Assuming RH, several authors have made improvements upon~\eqref{eq:kfree}, e.g., Montgomery and Vaughan~\cite{MontyVaughanSquarefree} proved that
\begin{equation}
\label{eq:kfreeMV}
Q_k(x) = \frac{x}{\zeta(k)} + O\left(x^{\frac{1}{k+1}+\varepsilon}\right).
\end{equation}
It is conjectured that~\eqref{eq:kfreeMV} is true with $2k$ instead of $k+1$, see~\cite{MontyVaughanSquarefree} for results prior to~\eqref{eq:kfreeMV}, and~\cite{MossinghoffOliveiraTrudgian} for an overview of recent advances.

We are going to state the main ideas in Montgomery and Vaughan's method with our explicit intentions in mind. Their proof starts with the observation that one can use~\eqref{eq:kfree} to write
\[
Q_{k}(x) = \sum_{\substack{n^{k}m\leq x \\ n\leq y}} \mu(n) + \sum_{\substack{n^{k}m\leq x \\ n>y}} \mu(n)
\]
for $1\leq y\leq x^{1/k}$. Denoting by $Q_{k,1}(x)$ and $Q_{k,2}(x)$ the above two sums, written in the same order, it is not hard to show by elementary methods that
\begin{equation}
\label{eq:Q1}
Q_{k,1}(x) = \sum_{n\leq y} \mu(n)\left\lfloor\frac{x}{n^k}\right\rfloor = \frac{x}{\zeta(k)} - \left(xf_{y}(k)+\frac{1}{2}M(y)+S_k(x,y)\right),
\end{equation}
where $M(x)$ is the Mertens function,
\begin{equation}
\label{eq:fy}
f_{y}(s) \de \sum_{y<n}\frac{\mu(n)}{n^s}
\end{equation}
for $\sigma>1$, and
\[
S_k(x,y)\de \sum_{n\leq y} \left(\frac{x}{n^k}-\left\lfloor{\frac{x}{n^k}}\right\rfloor-\frac{1}{2}\right)\mu(n).
\]
Trivially, $\left|S_k(x,y)\right|\leq y/2$. On the other hand, partial summation and~\eqref{eq:mobiusDef} implies conditional estimate $\left|f_y(s)\right|\ll_{\varepsilon} y^{1/2-\sigma+\varepsilon}$, see Lemma~\ref{lem:fy} for the precise explicit result. Concerning the second term $Q_{k,2}(x)$, the truncated Perron's formula is used together with
\begin{equation}
\label{eq:zetafy}
\zeta(s)f_{y}(ks) = \sum_{m=1}^{\infty}\sum_{y<n}\frac{\mu(n)}{\left(n^{k}m\right)^s},
\end{equation}
which is valid for $\sigma>1$, to obtain
\begin{equation}
\label{eq:PerronNonexplicit2}
Q_{k,2}(x) = \frac{1}{2\pi\ie}\int_{c-\ie T}^{c+\ie T}\frac{\zeta(s)f_{y}(ks)x^s\dif{s}}{s} + O\left(\frac{yx^{c}}{(c-1)T}\right) + O\left(\frac{yx\log{x}}{T}\right) + O(y)
\end{equation}
for $c>1$ and $T\geq 1$. In~\cite{MontyVaughanSquarefree} the choice $c=1+1/\log{x}$ and $T=x$ is proposed, so the error term in~\eqref{eq:PerronNonexplicit2} becomes $\ll y\log{x}$. RH allows to move the line of integration to the left, and taking $y=x^{1/(k+1)}$ will then produce~\eqref{eq:kfreeMV}. We should remark here that it seems like it is assumed in~\cite{MontyVaughanSquarefree} that coefficients of the Dirichlet series for~\eqref{eq:zetafy} are bounded by some absolute constant, resulting into $x^{\varepsilon}$, $\varepsilon>0$, for the error terms in~\eqref{eq:PerronNonexplicit2}. Although this claim does not change the final result~\eqref{eq:kfreeMV}, we would like to demonstrate that it is not correct. Note also that a sketch of the proof in~\cite[pp.~446--447]{MontgomeryVaughan} is different and avoids~\eqref{eq:PerronNonexplicit2}.

\begin{lemma}
\label{lem:zetafy}
Let $\sigma>1$, $y\geq 1$, and $k\geq 2$ be an integer. Assume that $\sum_{d=1}^{\infty}a_{d}d^{-s}$ is the Dirichlet series for $\zeta(s)f_{y}(ks)$. Then $\left|a_d\right|\leq y$.
\end{lemma}

\begin{proof}
Observe that
\[
f_{y}(ks) = \sum_{n^{1/k}\in\N_{>y}} \frac{\mu\left(n^{1/k}\right)}{n^s}.
\]
Then one can use Dirichlet convolution to obtain
\[
a_d = \sum_{\substack{q|d \\ q^{1/k}\in\N_{>y}}} \mu\left(q^{1/k}\right).
\]
Trivially, $a_d=0$ if $d$ is $k$-free number. Take $d=\left(p_1^{\nu_1}\cdots p_{l}^{\nu_l}\right)^{k} m$, where $p_1,\ldots,p_l$ are distinct prime numbers and $m$ is $k$-free number. Then
\begin{flalign*}
a_d &= -\sum_{\substack{y<p_{j_1} \\ 1\leq j_1\leq l}} 1 + \sum_{\substack{y<p_{j_1}p_{j_2} \\ 1\leq j_1<j_2\leq l}} 1 - \sum_{\substack{y<p_{j_1}p_{j_2}p_{j_3} \\ 1\leq j_1<j_2<j_3\leq l}} 1 + \cdots
+ (-1)^{l}\sum_{\substack{y<p_{j_1}\cdots p_{j_l} \\ 1\leq j_1<\cdots<j_l\leq l}} 1 \\
&= -1 + \sum_{\substack{p_{j_1}\leq y \\ 1\leq j_1\leq l}} 1 - \sum_{\substack{p_{j_1}p_{j_2}\leq y \\ 1\leq j_1<j_2\leq l}} 1 + \sum_{\substack{p_{j_1}p_{j_2}p_{j_3}\leq y \\ 1\leq j_1<j_2<j_3\leq l}} 1 - \cdots - (-1)^{l}\sum_{\substack{p_{j_1}\cdots p_{j_l}\leq y \\ 1\leq j_1<\cdots<j_l\leq l}} 1.
\end{flalign*}
It follows that $\left|a_d\right|\leq 1+(y-1)=y$.
\end{proof}

Let $y$ be large enough and let $p_1,\ldots,p_l$ be all primes in the interval $\left(\sqrt{y},y\right]$. If $d=\left(p_1\cdots p_l\right)^{k}m$ for $m\in\N$, then the proof of Lemma~\ref{lem:zetafy} implies
\[
a_d = -1+l = -1+\pi\left(y\right)-\pi\left(\sqrt{y}\right) \gg \frac{y}{\log{y}}.
\]
This shows that the coefficients $a_d$ from the Dirichlet series for~\eqref{eq:zetafy} are not bounded by some absolute constant.

The outline of our proof of~\eqref{eq:kfreeMV} is now very similar to the proof of~\eqref{eq:TitchMobius}. We choose $c=1+1/\log{x}$ and $T=ex$, see Corollary~\ref{cor:Perron2}. Then we move the line of integration to an arbitrary $\sigma'\in(1/2,3/4]$ and derive bounds for the resulting integrals by using~\eqref{eq:corMain2} and Lemmas~\ref{lem:zetabound2} and~\ref{lem:fy}. Theorem~\ref{thm:MainkFree} will follow after observation that one can take $\sigma'\to 1/2$ in the final inequalities. Because we are using also Theorem~\ref{thm:MainMoebius}, the proof of~\eqref{eq:thmSquarefree} consists of optimizing $\sigma$ and $t_0$ from~\eqref{eq:corMain} while taking $t_0=2\gamma_1$ in~\eqref{eq:corMain2} in order to obtain desired values for $\beta$.

\section{Proof of Theorem~\ref{thm:main}}
\label{sec:MainBound}

Let $N(T)$ be the number of the nontrivial zeros $\rho=\beta+\ie\gamma$ of $\zeta(s)$ with $0<\gamma\leq T$. The Riemann--von Mangoldt formula asserts that
\begin{equation}
\label{eq:RvM}
N(T) = \frac{T}{2\pi}\log{\frac{T}{2\pi e}} + \frac{7}{8} + Q(T),
\end{equation}
where $Q(T)\de S(T)+R(T)$ with $S(T)\ll \log{T}$ and $R(T)\ll 1/T$, see~\cite[Section 9.3]{Titchmarsh} for details. By~\cite[Lemma 2]{BrentPlattTrudgian} it is known that $\left|R(T)\right|\leq \frac{1}{150T}$ for $T\geq 2\pi$. As usual, $\pi S(t)$ is the argument of $\zeta(s)$ on the critical line. Closely related function is
\[
S_1(t)\de \int_{0}^{t}S(u)\dif{u}
\]
for $t\geq0$. Unconditionally we also have $S_1(t)\ll\log{t}$. However, on RH better estimates are known and recently effective bounds were provided. Let
\[
\mathcal{M}(a,b,c;t)\de a + \frac{b}{\left(\log{t}\right)^{c}\log\log{t}}
\]
for some positive real numbers $a$, $b$ and $c$, and define
\begin{gather}
\mathcal{M}_1(t)\de \mathcal{M}(0.759282,20.1911,0.285;t), \label{eq:M1} \\
\mathcal{M}_2(t)\de \mathcal{M}(0.653,60.12,0.2705;t). \label{eq:M2}
\end{gather}
By~\cite{SimonicSonRH} we know that RH implies
\[
\left|S(t)\right|\leq \phi_1(t)\frac{\log{t}}{\log{\log{t}}}
\]
with
\begin{equation}
\label{eq:SExpl}
\phi_1(t) \de \left\{
\begin{array}{ll}
  0.96, & 2\pi\leq t < 10^{2465}, \\
  \mathcal{M}_1(t)+0.96-\mathcal{M}_1\left(10^{2465}\right), & t\geq 10^{2465},
\end{array}
\right.
\end{equation}
and
\[
\left|S_1(t)\right|\leq \phi_2(t)\frac{\log{t}}{\left(\log{\log{t}}\right)^2},
\]
with
\begin{equation}
\label{eq:S1Expl}
\phi_2(t) \de \left\{
\begin{array}{ll}
  2.491, & 2\pi\leq t < 10^{208}, \\
  \mathcal{M}_2(t)+2.491-\mathcal{M}_2\left(10^{208}\right), & t\geq 10^{208}.
\end{array}
\right.
\end{equation}
Observe that $\phi_1(t)$ and $\phi_2(t)$ are continuous and decreasing functions for $t\geq 2\pi$. Here some improvements may be possible. However, we should emphasize that for our proof to work we need to know analytic properties of some functions which include $\phi_1(t)$ and $\phi_2(t)$, e.g., that $\phi_2(t)\left(\log{\log{t}}\right)^{-2}\log{t}$ is an increasing function for $t\geq\exp{\left(e^2\right)}$.

The main result of this section is Theorem~\ref{thm:zetaMain} which its Corollary~\ref{cor:main} immediately implies Theorem~\ref{thm:main}. We are following~\cite{TitchConseq}. In the literature one can find two similar proofs of~\eqref{eq:TitchZetaBound} and~\eqref{eq:TitchZetaBound2}, namely~\cite[Theorem 14.14 (B)]{Titchmarsh} and~\cite[Theorem 13.23]{MontgomeryVaughan}. The former proof is closer to~\cite{TitchConseq} in the sense that it also relies on conditional estimates for $S(t)$ and $S_1(t)$, while in the latter proof these functions implicitly appear through properties of $\zeta'(s)/\zeta(s)$. We would also like to note that there is an error in~\cite[Section 14.10, p.~347]{Titchmarsh}; it is true that RH implies
\begin{equation}
\label{eq:problem}
\int_{\alpha+\ie T}^{2+\ie T}\frac{\log{\zeta(z)}}{z-s}\dif{z} = O\left(\frac{\log{T}}{t-T}\right),
\end{equation}
uniformly for $1/2<\alpha<2$ and $T<t$, but~\eqref{eq:problem} does not follow from~\eqref{eq:someEstimateForZeta} because this estimate is not uniform in $\sigma$. Instead one should consider the (unconditional) estimate
\[
\log{\zeta(s)} = \sum_{\left|t-\gamma\right|\leq 1} \log{\left(s-\rho\right)} + O\left(\log{t}\right),
\]
which is uniform in $\sigma\in[-1,2]$, see~\cite[Theorem 9.6 (B)]{Titchmarsh}. On RH we then have
\[
\int_{\alpha}^{2}\left|\log{\zeta\left(u+\ie T\right)}\right|\dif{u} \leq \sum_{\left|t-\gamma\right|\leq 1} \int_{\alpha}^{2} \left|\log{\left(u-\frac{1}{2}+\ie\left(T-\gamma\right)\right)}\right|\dif{u} + O\left(\log{T}\right),
\]
and the right-hand side of the above inequality can be easily seen $\ll \log{T}$.

The main idea in Titchmarsh's older proof is to write $\log{\zeta(s)}$ as the integral of $Q(u)$. This is achieved by using Hadamard's factorization theorem for $\xi\left(1/2+\ie z\right)$, the Riemann--von Mangoldt formula~\eqref{eq:RvM}, and also Stirling's formula, see the following lemma.

\begin{lemma}
\label{lem:logzeta}
Assume the Riemann Hypothesis. Let $s=\sigma+\ie t$ with $1/2<\sigma\leq 3/2$ and $t\geq t_0>\gamma_1$. Then
\begin{equation}
\label{eq:logzeta}
\log{\left|\zeta(s)\right|} = \Re\left\{2\left(s-\frac{1}{2}\right)^2\int_{\gamma_1}^{\infty}\frac{Q(u)}{u\left(u^2+\left(s-\frac{1}{2}\right)^2\right)}\dif{u}\right\} + R_1,
\end{equation}
where
\begin{flalign}
\label{eq:R1}
\left|R_1\right| \leq \mathcal{R}_1\left(t_0\right) &\de \left|\log{\left|\xi\left(\frac{1}{2}\right)\right|}\right| + \frac{1}{4}\log{\frac{2e}{\pi}} + \frac{\delta\left(t_0\right)}{\pi}\int_{0}^{\gamma_1}\frac{\left|\log{\left(2\pi e/u\right)}\right|}{1-\left(u/t_0\right)^2}\dif{u} \nonumber \\
&+ \frac{7}{8}\left|\log{\left(\frac{1}{\gamma_1^2}-\frac{1}{t_0^2}\right)}\right| + \frac{7+2t_0}{4}\log{\frac{2t_0}{2t_0-1}} + \frac{1}{6t_0} + \frac{4}{45t_0^3}
\end{flalign}
and
\begin{equation}
\label{eq:delta}
\delta\left(t_0\right)\de \left(1+\frac{1}{t_0}\right)^2.
\end{equation}
\end{lemma}

\begin{proof}
Define $z\de -\ie\left(s-1/2\right)$ and take
\begin{equation}
\label{eq:Xi}
\Xi(z)\de\xi\left(\frac{1}{2}+\ie z\right), \quad \xi(s) \de \frac{1}{2}s(s-1)\pi^{-\frac{s}{2}}\Gamma\left(\frac{s}{2}\right)\zeta(s).
\end{equation}
Let $0<\gamma_1\leq\gamma_2\leq\cdots\leq\gamma_n\leq\cdots$ denote the ordinates of the nontrivial zeros in the upper half-plane. By Hadamard's factorization theorem it follows that
\[
\Xi(z) = \xi\left(\frac{1}{2}\right)\prod_{n=1}^{\infty}\left(1-\frac{z^2}{\gamma_n^2}\right).
\]
Therefore,
\begin{equation}
\label{eq:logXi1}
\log{\left|\Xi(z)\right|} = \log{\left|\xi\left(\frac{1}{2}\right)\right|} + \Re\left\{I_1\right\}, \quad
I_1 \de \int_{\gamma_1}^{\infty}\frac{2z^2N(u)}{u\left(z^2-u^2\right)}\dif{u},
\end{equation}
see~\cite[p.~249]{TitchConseq} for details. Writing
\begin{gather*}
O_1 \de -\frac{z^2}{\pi}\int_{0}^{\gamma_1}\frac{\log{u}}{z^2-u^2}\dif{u} + \frac{z^2\log{\left(2\pi e\right)}}{\pi}\int_{0}^{\gamma_1}\frac{\dif{u}}{z^2-u^2}, \\
O_2 \de - \frac{7}{4}\log{z} + \frac{7}{4}\int_{\gamma_1}^{\infty}\frac{z^2\dif{u}}{u\left(z^2-u^2\right)},
\end{gather*}
by~\eqref{eq:RvM} it follows that
\begin{equation}
\label{eq:I1}
I_1 = \int_{\gamma_1}^{\infty}\frac{2z^2 Q(u)}{u\left(z^2-u^2\right)}\dif{u} + \frac{\ie}{2}z\log{z} - \frac{\pi}{4}z - \frac{\ie \log{\left(2\pi e\right)}}{2}z + \frac{7}{4}\log{z} + O_1 + O_2,
\end{equation}
see~\cite[p.~249]{TitchConseq} for details. Because
\begin{equation}
\label{eq:boundsz}
t\leq |z|\leq t\left(1+\frac{1}{t_0}\right), \quad \left|z^2-u^2\right|\geq \left|u^2-t^2\right|,
\end{equation}
we have
\[
\left|\Re\left\{O_1\right\}\right| \leq \left|O_1\right| \leq \frac{1}{\pi}\delta\left(t_0\right)\int_{0}^{\gamma_1}\frac{\left|\log{\left(2\pi e/u\right)}\right|}{1-\left(u/t_0\right)^2}\dif{u}.
\]
Also,
\[
\int_{\gamma_1}^{\infty}\frac{z^2\dif{u}}{u\left(z^2-u^2\right)} = \frac{1}{2}\log{\left(1-\frac{z^2}{\gamma_1^2}\right)},
\]
which implies
\[
\left|\Re\left\{O_2\right\}\right| = \frac{7}{8}\left|\log{\left|\frac{1}{z^2}-\frac{1}{\gamma_1^2}\right|}\right| \leq \frac{7}{8}\left|\log{\left(\frac{1}{\gamma_1^2}-\frac{1}{t_0^2}\right)}\right|.
\]
On the other hand, by \eqref{eq:Xi} we have
\begin{multline}
\label{eq:logXi12}
\log{\left|\Xi(z)\right|} = \log{\frac{1}{2}} - \frac{\log{\pi}}{4} + \log{\left|\frac{1}{2}+\ie z\right|} + \log{\left|-\frac{1}{2}+\ie z\right|} \\
- \Re\left\{\ie\frac{\log{\pi}}{2}z\right\} + \log{\left|\Gamma\left(\frac{1}{4}+\ie\frac{z}{2}\right)\right|} + \log{\left|\zeta(s)\right|}.
\end{multline}
We can write
\[
\log{\left(\frac{1}{2}+\ie z\right)} = \log{z} + \frac{\pi}{2}\ie + O_3(z), \quad O_3(z)\de \log{\left(1+\frac{1}{2\ie z}\right)}.
\]
Taking $w=1/\left(2\ie z\right)$, we obtain
\begin{gather*}
\left|O_3(z)\right| \leq \log{\frac{1}{1-|w|}} \leq \log{\frac{2t_0}{2t_0-1}}, \\
\left|zO_3(z)\right|\leq \frac{1}{2|w|}\log{\frac{1}{1-|w|}}\leq t_0\log{\frac{2t_0}{2t_0-1}}
\end{gather*}
since $|w|\leq 1/\left(2t_0\right)$, and the former bound is also true for $\left|O_3(-z)\right|$. By Stirling's formula for $\log{\Gamma(z)}$ with an explicit error term, see~\cite[p.~294]{Olver}, equality \eqref{eq:logXi12} implies
\begin{flalign}
\label{eq:logXi2}
\log{\left|\Xi(z)\right|} &= \log{\left|\zeta(s)\right|} + \Re\left\{\frac{\ie}{2}z\log{z} - \frac{\pi}{4}z - \frac{\ie \log{\left(2\pi e\right)}}{2}z + \frac{7}{4}\log{z}\right\} - \frac{1}{4}\log{\frac{2e}{\pi}} \nonumber \\
&+ \Re\left\{\frac{3}{4}O_3(z) + O_3(-z) + \frac{\ie}{2}zO_3(z) + O_4\left(\frac{1}{4}+\ie \frac{z}{2}\right)\right\},
\end{flalign}
where
\[
\left|O_4\left(\frac{1}{4}+\ie \frac{z}{2}\right)\right| \leq \frac{1}{6t_0} + \frac{4}{45t_0^3}.
\]
Comparing \eqref{eq:logXi1} with \eqref{eq:logXi2} while using \eqref{eq:I1} gives \eqref{eq:logzeta} with
\begin{multline*}
R_1 = \log{\left|\xi\left(\frac{1}{2}\right)\right|} + \frac{1}{4}\log{\frac{2e}{\pi}} \\
+ \Re\left\{O_1 + O_2 - \frac{3}{4}O_3(z) - O_3(-z) - \frac{\ie}{2}zO_3(z) - O_4\left(\frac{1}{4}+\ie \frac{z}{2}\right)\right\}.
\end{multline*}
After collecting all bounds for the error terms, the final result easily follows.
\end{proof}

The next step in the proof is to restrict the range of integration in~\eqref{eq:logzeta} to some neighbourhood of $t$.

\begin{lemma}
\label{lem:logzeta2}
Assume the Riemann Hypothesis. Let $s=\sigma+\ie t$ with $1/2<\sigma\leq 3/2$, $t\geq\max\left\{2\exp{\left(e^2\right)},t_0\right\}$, $t_0\geq2\gamma_1$, and $0<\xi\leq t_0/2$. Then
\begin{equation}
\label{eq:logzeta2}
\log{\left|\zeta(s)\right|} = \Re\left\{2\left(s-\frac{1}{2}\right)^2\int_{t-\xi}^{t+\xi}\frac{Q(u)}{u\left(u^2+\left(s-\frac{1}{2}\right)^2\right)}\dif{u}\right\} + R_2,
\end{equation}
where
\begin{flalign}
\label{eq:R2}
\left|R_2\right| &\leq \delta\left(t_0\right)\left(\frac{85}{4}+3\delta\left(t_0\right)\right)\frac{\phi_2(3t/2)\log{(3t/2)}}{\xi\left(\log{\log{(3t/2)}}\right)^2} \nonumber \\
&+ 2\delta\left(t_0\right)\left(4+\left(3+\delta\left(t_0\right)\right)\left(\frac{t_0}{2e^{e^2}}+\frac{3+e}{2e}\right)\right)\frac{\phi_2(t)\log{t}}{\xi\left(\log{\log{t}}\right)^2} \nonumber \\
&+ \frac{\delta\left(t_0\right)}{75}\left(1+\frac{1}{e}+\frac{t_0}{2\gamma_1}\right)\frac{1}{\xi} + \frac{2\left(1+t_0\right)^{2}\left|S_1\left(\gamma_1\right)\right|}{\gamma_1\left(t_0^2-\gamma_1^2\right)} + 0.81\delta\left(t_0\right)\left(3+\delta\left(t_0\right)\right) \nonumber \\
&+ \mathcal{R}_1\left(t_0\right) + \frac{4\delta\left(t_0\right)}{5e}\left(3+\frac{\delta\left(t_0\right)}{4}\right)\frac{\phi_2(3t/2)}{\left(\log{\log{(3t/2)}}\right)^2},
\end{flalign}
$\delta\left(t_0\right)$ is defined by~\eqref{eq:delta} and $\mathcal{R}_1\left(t_0\right)$ is from~\eqref{eq:R1}.
\end{lemma}

\begin{proof}
Remembering that $Q(u)=S(u)+R(u)$, we can write
\begin{multline}
\label{eq:settingLemLogzeta2}
\int_{\gamma_1}^{\infty}\frac{Q(u)\dif{u}}{u\left(u^2+\left(s-\frac{1}{2}\right)^2\right)} = \int_{t-\xi}^{t+\xi}\frac{Q(u)\dif{u}}{u\left(u^2+\left(s-\frac{1}{2}\right)^2\right)} \\
+ \left(\int_{\gamma_1}^{t-\xi}+\int_{t+\xi}^{\infty}\right)\frac{S(u)+R(u)}{u\left(u^2+\left(s-\frac{1}{2}\right)^2\right)}\dif{u}.
\end{multline}
The idea is to apply Lemma~\ref{lem:logzeta} on~\eqref{eq:settingLemLogzeta2}. We will estimate the modulus of the last two integrals by separating two cases which correspond to functions $S(u)$ and $R(u)$. Note that $t-\xi\geq t/2$ and $t+\xi\leq 3t/2$. We are also using estimates~\eqref{eq:boundsz}.

In the case of $R(u)$, we obtain
\begin{equation}
\label{eq:intR}
\left|\left(\int_{\gamma_1}^{t-\xi}+\int_{t+\xi}^{\infty}\right)\frac{R(u)\dif{u}}{u\left(u^2+\left(s-\frac{1}{2}\right)^2\right)}\right| \leq \frac{1}{150}\left(1+\frac{1}{e}+\frac{t_0}{2\gamma_1}\right)\frac{1}{t^2\xi},
\end{equation}
because
\[
0 < \int_{\gamma_1}^{t-\xi}\frac{\dif{u}}{u^2\left(1-\left(\frac{u}{t}\right)^2\right)} + \int_{t+\xi}^{\infty}\frac{\dif{u}}{u^2\left(\left(\frac{u}{t}\right)^2-1\right)} \leq \left(1+\frac{1}{e}+\frac{\xi}{\gamma_1}\right)\frac{1}{\xi}.
\]
The last inequality follows by exact integration, and by the inequalities $\log{x}\leq x/e$ and $\log{(1+x)}\leq x$ which are valid for $x>0$. Such simple global inequalities are good enough for our purpose.

In the case of $S(t)$, we will consider each integral in brackets in~\eqref{eq:settingLemLogzeta2} separately. Integration by parts implies
\begin{flalign}
\label{eq:int1}
\int_{\gamma_1}^{t-\xi}\frac{S(u)\dif{u}}{u\left(u^2+\left(s-\frac{1}{2}\right)^2\right)} &= \frac{S_1\left(t-\xi\right)}{\left(t-\xi\right)\left(\left(t-\xi\right)^2+\left(s-\frac{1}{2}\right)^2\right)} - \frac{S_1\left(\gamma_1\right)}{\gamma_1\left(\gamma_1^2+\left(s-\frac{1}{2}\right)^2\right)} \nonumber \\
&+\int_{\gamma_1}^{t-\xi}\frac{3u^2+\left(s-\frac{1}{2}\right)^2}{u^2\left(u^2+\left(s-\frac{1}{2}\right)^2\right)^2}S_1(u)\dif{u}.
\end{flalign}
Note that $\phi_{j}(u)\left(\log{\log{u}}\right)^{-j}\log{u}$ is for $j\in\{1,2\}$ and $u\geq e^{e^2}$ an increasing function. We are going to split the range of integration in the second integral in~\eqref{eq:int1} into two parts:
\begin{multline}
\label{eq:int111}
\left|\int_{\gamma_1}^{e^{e^2}}\frac{3u^2+\left(s-\frac{1}{2}\right)^2}{u^2\left(u^2+\left(s-\frac{1}{2}\right)^2\right)^2}S_1(u)\dif{u}\right| \leq
\frac{1}{t^2}\left(3+\delta\left(t_0\right)\right)\times \\
\times\int_{\gamma_1}^{e^{e^2}}\frac{2.491\log{u}\dif{u}}{u^2\left(1-\left(\frac{u}{2e^{e^2}}\right)^2\right)^2\left(\log{\log{u}}\right)^2} \leq
\frac{0.405\left(3+\delta\left(t_0\right)\right)}{t^2}
\end{multline}
and
\begin{multline}
\label{eq:int11}
\left|\int_{e^{e^2}}^{t-\xi}\frac{3u^2+\left(s-\frac{1}{2}\right)^2}{u^2\left(u^2+\left(s-\frac{1}{2}\right)^2\right)^2}S_1(u)\dif{u}\right| \leq \frac{\left(3+\delta\left(t_0\right)\right)\phi_2(t)\log{t}}{t^2\left(\log{\log{t}}\right)^2}\times \\
\times\int_{e^{e^2}}^{t-\xi}\frac{\dif{u}}{u^2\left(1-\left(\frac{u}{t}\right)^2\right)^2} \leq
\left(3+\delta\left(t_0\right)\right)\left(\frac{t_0}{2e^{e^{2}}}+\frac{1}{2}+\frac{3}{2e}\right)\frac{\phi_2(t)\log{t}}{t^2\xi\left(\log{\log{t}}\right)^2}.
\end{multline}
Observe that we used~\eqref{eq:S1Expl} in estimation~\eqref{eq:int111}. Also,
\begin{gather}
\frac{\left|S_1\left(t-\xi\right)\right|}{\left(t-\xi\right)\left|\left(t-\xi\right)^2+\left(s-\frac{1}{2}\right)^2\right|} \leq \left(\frac{t}{t-\xi}\right)^{2}\frac{\phi_2(t)\log{t}}{t^{2}\xi\left(\log{\log{t}}\right)^2} \leq \frac{4\phi_2(t)\log{t}}{t^{2}\xi\left(\log{\log{t}}\right)^2}, \label{eq:int12} \\
\frac{\left|S_1\left(\gamma_1\right)\right|}{\gamma_1\left|\gamma_1^2+\left(s-\frac{1}{2}\right)^2\right|} \leq \frac{t_0^2\left|S_1\left(\gamma_1\right)\right|}{t^2\gamma_1\left(t_0^2-\gamma_1^2\right)}. \label{eq:int13}
\end{gather}
Considering the second integral in brackets in~\eqref{eq:settingLemLogzeta2}, integration by parts implies
\begin{flalign}
\label{eq:int2}
\int_{t+\xi}^{\infty}\frac{S(u)\dif{u}}{u\left(u^2+\left(s-\frac{1}{2}\right)^2\right)} &= -\frac{S_1\left(t+\xi\right)}{\left(t+\xi\right)\left(\left(t+\xi\right)^2+\left(s-\frac{1}{2}\right)^{2}\right)} \nonumber \\
&+ \left(\int_{t+\xi}^{3t/2}+\int_{3t/2}^{\infty}\right)\frac{3u^2+\left(s-\frac{1}{2}\right)^2}{u^2\left(u^2+\left(s-\frac{1}{2}\right)^2\right)^2}S_1(u)\dif{u}.
\end{flalign}
Similarly as before, bounds on the moduli of the last two integrals in~\eqref{eq:int2} are
\begin{multline}
\label{eq:int21}
\left|\int_{t+\xi}^{3t/2}\frac{3u^2+\left(s-\frac{1}{2}\right)^2}{u^2\left(u^2+\left(s-\frac{1}{2}\right)^2\right)^2}S_1(u)\dif{u}\right| \leq
\frac{\left(\frac{27}{4}+\delta\left(t_0\right)\right)\phi_2(3t/2)\log{(3t/2)}}{t^2\left(\log{\log{(3t/2)}}\right)^2}\times \\
\times\int_{t+\xi}^{3t/2}\frac{\dif{u}}{u^2\left(\left(\frac{u}{t}\right)^2-1\right)^2} \leq \frac{3}{2}\left(\frac{27}{4}+\delta\left(t_0\right)\right)\frac{\phi_2(3t/2)\log{(3t/2)}}{t^2\xi\left(\log{\log{(3t/2)}}\right)^2}
\end{multline}
and, because $\phi_2(u)\left(\log{\log{u}}\right)^{-2}$ is for $u\geq2\pi$ a decreasing function,
\begin{multline}
\label{eq:int22}
\left|\int_{3t/2}^{\infty}\frac{3u^2+\left(s-\frac{1}{2}\right)^2}{u^2\left(u^2+\left(s-\frac{1}{2}\right)^2\right)^2}S_1(u)\dif{u}\right| \leq
\frac{1}{e}\left(3+\frac{\delta\left(t_0\right)}{4}\right)\frac{\phi_2(3t/2)}{\left(\log{\log{(3t/2)}}\right)^2}\times \\ \times\int_{3t/2}^{\infty}\frac{u\dif{u}}{\left(u^2-t^2\right)^2} \leq \frac{2}{5e}\left(3+\frac{\delta\left(t_0\right)}{4}\right)\frac{\phi_2(3t/2)}{t^2\left(\log{\log{(3t/2)}}\right)^2}.
\end{multline}
In~\eqref{eq:int22} we used $\log{u}\leq u/e$. Also,
\begin{equation}
\label{eq:int23}
\frac{\left|S_1\left(t+\xi\right)\right|}{\left(t+\xi\right)\left|\left(t+\xi\right)^2+\left(s-\frac{1}{2}\right)^{2}\right|} \leq \frac{\phi_2(3t/2)\log{(3t/2)}}{2t^2\xi\left(\log{\log{(3t/2)}}\right)^2}.
\end{equation}
Multiplying~\eqref{eq:settingLemLogzeta2} by $2\left(s-1/2\right)^2$, applying Lemma~\ref{lem:logzeta}, and using bounds~\eqref{eq:int111}, \eqref{eq:int11}, \eqref{eq:int12}, \eqref{eq:int13} in~\eqref{eq:int1}, and~\eqref{eq:int21}, \eqref{eq:int22}, \eqref{eq:int23} in~\eqref{eq:int2}, and~\eqref{eq:intR}, finally gives~\eqref{eq:logzeta2} with~\eqref{eq:R2}. The proof of Lemma~\ref{lem:logzeta2} is thus complete.
\end{proof}

\begin{theorem}
\label{thm:zetaMain}
Assume the Riemann Hypothesis. Let $s=\sigma+\ie t$ with $1/2<\sigma\leq 3/2$, $t\geq\max\left\{2\exp{\left(e^2\right)},t_0\right\}$, $t_0\geq2\gamma_1$, and $0<\lambda\leq \left(t_0/2\right)\log{\log{\left(3t_0/2\right)}}$. Then
\begin{flalign}
\label{eq:abslogzeta}
\left|\log{\left|\zeta(s)\right|}\right| &\leq 2\left(\frac{\phi_1(3t/2)\log{(3t/2)}}{\log{\log{(3t/2)}}}+\frac{1}{75t}\right)\log{\left(1+\frac{2\lambda}{\left(\sigma-\frac{1}{2}\right)\log{\log{(3t/2)}}}\right)} \nonumber \\
&+ \Omega\left(t_0,\lambda,t\right)
\end{flalign}
and
\begin{equation}
\label{eq:abszeta}
\log{\left|\zeta(s)\right|} \leq \frac{\lambda\log{(3t/2)}}{\pi\log{\log{(3t/2)}}} + \Omega\left(t_0,\lambda;t\right),
\end{equation}
where
\begin{flalign*}
\Omega\left(t_0,\lambda;t\right) &\de \frac{a_1\left(t_0\right)\phi_2(3t/2)\log{(3t/2)}}{\lambda\log{\log{(3t/2)}}} + \frac{a_2\left(t_0\right)\log{\log{(3t/2)}}}{\lambda} + a_3\left(t_0\right)\\
&+ 12\left(\frac{\phi_1\left(t/2\right)\log{\left(t/2\right)}}{t\log{\log{\left(t/2\right)}}}+\frac{1}{75t^2}\right)
\frac{\lambda}{\log{\log{(3t/2)}}} + \frac{a_4\left(t_0\right)\phi_2(3t/2)}{\left(\log{\log{(3t/2)}}\right)^2},
\end{flalign*}
and $\phi_1$ and $\phi_2$ are defined by~\eqref{eq:SExpl} and~\eqref{eq:S1Expl}, respectively,
\begin{equation}
\label{eq:a1}
a_1\left(t_0\right)\de \delta\left(t_0\right)\left(\frac{117}{4}+3\delta\left(t_0\right)
+\left(3+\delta\left(t_0\right)\right)\left(\frac{t_0}{e^{e^2}}+\frac{3+e}{e}\right)\right),
\end{equation}
\begin{equation}
\label{eq:a2}
a_2\left(t_0\right)\de \frac{\delta\left(t_0\right)}{75}\left(\frac{t_0}{2\gamma_1}+1+\frac{1}{e}\right),
\end{equation}
\begin{equation}
\label{eq:a3}
a_3\left(t_0\right)\de \frac{3\left(1+t_0\right)^{2}}{\gamma_1\left(t_0^2-\gamma_1^2\right)} + 0.81\delta\left(t_0\right)\left(3+\delta\left(t_0\right)\right) + \mathcal{R}_1\left(t_0\right) ,
\end{equation}
\begin{equation}
\label{eq:a4}
a_4\left(t_0\right)\de \frac{4\delta\left(t_0\right)}{5e}\left(3+\frac{\delta\left(t_0\right)}{4}\right),
\end{equation}
$\delta\left(t_0\right)$ is defined by~\eqref{eq:delta} and $\mathcal{R}_1\left(t_0\right)$ is from~\eqref{eq:R1}.
\end{theorem}

\begin{proof}
Take $\xi\de\lambda/\log{\log{(3t/2)}}$. Then $0<\xi\leq t_0/2$, and also $t-\xi\geq t/2$ and $t+\xi\leq 3t/2$. Firstly we are going to prove~\eqref{eq:abslogzeta}. By Lemma~\ref{lem:logzeta2} we thus have
\begin{equation}
\label{eq:R331}
\log{\left|\zeta(s)\right|} = -\Re\left\{\int_{t-\xi}^{t+\xi}\frac{Q(u)\dif{u}}{u+\ie\left(s-\frac{1}{2}\right)}\right\} + R_3,
\end{equation}
where
\[
R_3 \de R_2 - \Re\left\{\int_{t-\xi}^{t+\xi}\frac{Q(u)\dif{u}}{u-\ie\left(s-\frac{1}{2}\right)}\right\} + 2\int_{t-\xi}^{t+\xi}\frac{Q(u)}{u}\dif{u}
\]
and the real term $R_2$ is from~\eqref{eq:logzeta2}. Because
\[
\left|\Re\left\{\frac{Q(u)}{u-\ie\left(s-\frac{1}{2}\right)}\right\}\right| = \frac{(u+t)\left|Q(u)\right|}{(u+t)^2+\left(\sigma-\frac{1}{2}\right)^2} \leq \frac{\left|Q(u)\right|}{u+t} \leq \frac{\left|Q(u)\right|}{u}
\]
for $u\in\left[t-\xi,t+\xi\right]$, and $\left|Q(u)\right|\leq \left|S(u)\right| + \left|R(u)\right|$, it follows
\begin{equation}
\label{eq:R3}
\left|R_3\right| \leq \left|R_2\right| + 3\int_{t-\xi}^{t+\xi}\frac{\left|Q(u)\right|}{u}\dif{u} \leq \left|R_2\right| + 12\left(\frac{\phi_1\left(t/2\right)\log{\left(t/2\right)}}{t\log{\log{\left(t/2\right)}}}+\frac{1}{75t^2}\right)\xi
\end{equation}
since $\phi_1(u)\left(u\log{\log{u}}\right)^{-1}\log{u}$ is for $u\geq 2\pi$ a decreasing function. Therefore, inequality~\eqref{eq:R3} asserts that $\left|R_3\right| \leq \Omega\left(t_0,\lambda;t\right)$.

Also,
\begin{flalign}
\label{eq:R31}
\left|\Re\left\{\int_{t-\xi}^{t+\xi}\frac{Q(u)\dif{u}}{u+\ie\left(s-\frac{1}{2}\right)}\right\}\right| &\leq \int_{t-\xi}^{t+\xi}\frac{\left|Q(u)\right|\dif{u}}{\sqrt{\left(\sigma-\frac{1}{2}\right)^2+(u-t)^2}} \nonumber \\
&\leq 2\left(\frac{\phi_1(3t/2)\log{(3t/2)}}{\log{\log{(3t/2)}}}+\frac{1}{75t}\right)\int_{0}^{\xi}\frac{\dif{u}}{\sqrt{\left(\sigma-\frac{1}{2}\right)^2+u^2}} \nonumber \\
&\leq 2\left(\frac{\phi_1(3t/2)\log{(3t/2)}}{\log{\log{(3t/2)}}}+\frac{1}{75t}\right)\log\left(1+\frac{2\xi}{\sigma-\frac{1}{2}}\right).
\end{flalign}
Taking~\eqref{eq:R3}, together with~\eqref{eq:R2} and $\left|S_1\left(\gamma_1\right)\right|\leq 1.5$, and~\eqref{eq:R31} as an upper bound for the integral in~\eqref{eq:R331}, we obtain~\eqref{eq:abslogzeta}.

In order to prove~\eqref{eq:abszeta}, observe that
\begin{flalign*}
-\Re\left\{\int_{t-\xi}^{t+\xi}\frac{Q(u)\dif{u}}{u+\ie\left(s-\frac{1}{2}\right)}\right\} &= \int_{0}^{\xi}\frac{u\left(Q(t-u)-Q(t+u)\right)}{\left(\sigma-\frac{1}{2}\right)^2+u^2}\dif{u} \\
&\leq \frac{1}{\pi}\log{\frac{3t}{2}} \int_{0}^{\xi}\frac{u^2 \dif{u}}{\left(\sigma-\frac{1}{2}\right)^2+u^2} \leq \frac{\xi}{\pi}\log{\frac{3t}{2}}
\end{flalign*}
since
\begin{flalign*}
Q(t-u)-Q(t+u) &= N(t-u) - N(t+u) + \frac{t+u}{2\pi}\log{\frac{t+u}{2\pi e}} - \frac{t-u}{2\pi}\log{\frac{t-u}{2\pi e}} \\
&\leq 0 + \frac{u}{\pi}\log{\frac{t+u}{2\pi}} \leq \frac{u}{\pi}\log{\frac{3t}{2}}
\end{flalign*}
by~\eqref{eq:RvM} and the mean-value theorem. Inequality~\eqref{eq:abszeta} now follows from~\eqref{eq:R331}. The proof of Theorem~\ref{thm:zetaMain} is thus complete.
\end{proof}

\begin{corollary}
\label{cor:main}
Assume the Riemann Hypothesis. Let $s=\sigma+\ie t$ with $1/2<\sigma\leq 3/2$, and $2\gamma_1\leq t_0\leq 50$, $\lambda\left(t_0\right)\de \left(t_0/2\right)\log{\log{\left(3t_0/2\right)}}$ and
\begin{multline*}
\omega\left(t_0;u\right) \de \frac{a_1\left(t_0\right)\widehat{\phi}_2(u)}{\lambda\left(t_0\right)u} + \frac{a_2\left(t_0\right)u}{\lambda\left(t_0\right)e^{u}} + \frac{15.83}{e^{u}} \\
+ \frac{12\lambda\left(t_0\right)}{ue^{e^{u}}}\left(\frac{3\widehat{\phi}_1(u)}{2u}+\frac{3}{100e^{u+e^u}}\right) + \frac{a_4\left(t_0\right)\widehat{\phi}_2(u)}{u^2e^{u}},
\end{multline*}
\begin{gather}
\omega_1\left(\sigma,t_0;u\right) \de  2\left(\frac{\widehat{\phi}_1(u)}{u}+\frac{1}{50e^{u+e^{u}}}\right)\log{\left(1+\frac{2\lambda\left(t_0\right)}{\left(\sigma-\frac{1}{2}\right)u}\right)}
+ \omega\left(t_0;u\right), \label{eq:omega1} \\
\omega_2\left(t_0;u\right) \de \frac{\lambda\left(t_0\right)}{\pi u} + \omega\left(t_0;u\right), \label{eq:omega2}
\end{gather}
where $a_1\left(t_0\right)$, $a_2\left(t_0\right)$ and $a_4\left(t_0\right)$ are defined by~\eqref{eq:a1}, \eqref{eq:a2} and~\eqref{eq:a4}, respectively,
\[
\widehat{\phi}_1(u) \de \left\{
\begin{array}{ll}
  0.96, & \log\log{2\pi}\leq u < \log\log{10^{2465}}, \\
  1.719282-\mathcal{M}_1\left(10^{2465}\right)+\frac{20.1911}{ue^{0.285u}}, & u\geq \log\log{10^{2465}},
\end{array}
\right.
\]
and
\[
\widehat{\phi}_2(u) \de \left\{
\begin{array}{ll}
  2.491, & \log\log{2\pi}\leq u < \log\log{10^{208}}, \\
  3.144-\mathcal{M}_2\left(10^{208}\right)+\frac{60.12}{ue^{0.2705u}}, & u\geq \log\log{10^{208}},
\end{array}
\right.
\]
with $\mathcal{M}_1$ and $\mathcal{M}_2$ defined by~\eqref{eq:M1} and~\eqref{eq:M2}, respectively. Then $\omega_1\left(\sigma,t_0;u\right)$ and $\omega_2\left(t_0;u\right)$ are decreasing positive continuous functions for $u\geq1$ with
\begin{gather}
\lim_{u\to\infty}u\cdot\omega_1\left(\sigma,t_0;u\right) = \left(3.144-\mathcal{M}_2\left(10^{208}\right)\right)\frac{a_1\left(t_0\right)}{\lambda\left(t_0\right)}, \label{eq:limits} \\
\lim_{u\to\infty}u\cdot \omega_2\left(t_0;u\right) = \frac{1}{\pi}\lambda\left(t_0\right) + \left(3.144-\mathcal{M}_2\left(10^{208}\right)\right)\frac{a_1\left(t_0\right)}{\lambda\left(t_0\right)}, \label{eq:limits2}
\end{gather}
and $\omega_1\left(\sigma,t_0;u\right)$ is a decreasing function for $\sigma>1/2$. Moreover, for $2\exp{\left(e^2\right)}\leq t\leq T$ the inequalities~\eqref{eq:corMain} and~\eqref{eq:corMain2} are true.
\end{corollary}

\begin{proof}
It is clear that $\omega_1\left(\sigma,t_0;u\right)$ and $\omega_2\left(t_0;u\right)$ are decreasing positive continuous functions in the variable $u\geq 1$ since $\widehat{\phi}_1(u)$, $\widehat{\phi}_2(u)$, and $ue^{-u}$ are decreasing positive functions for $u\geq 1$, and that $\omega_1\left(\sigma,t_0;u\right)$ is a decreasing function in $\sigma>1/2$. It is also clear that~\eqref{eq:limits} and~\eqref{eq:limits2} hold. Take $u=\log{\log{\left(3t/2\right)}}$. Then $\widehat{\phi}_1(u)=\phi_1(3t/2)$ and $\widehat{\phi}_2(u)=\phi_2(3t/2)$, where $\phi_1$ and $\phi_2$ are defined by~\eqref{eq:SExpl} and~\eqref{eq:S1Expl}, respectively. Because the function $a_3\left(t_0\right)$, given by~\eqref{eq:a3}, is decreasing, it follows that $a_3\left(t_0\right)\leq 15.83$. Inequalities~\eqref{eq:corMain} and~\eqref{eq:corMain2} now easily follow from Theorem~\ref{thm:zetaMain}, if we are able to show that
\[
\widehat{\omega}_1(u) \de e^{u}\omega_1\left(\sigma,t_0;u\right), \quad \widehat{\omega}_2(u) \de e^{u}\omega_2\left(t_0;u\right)
\]
are increasing functions for $u\geq \log{\log{\left(3\exp{\left(e^2\right)}\right)}}$.

Straightforward calculation confirms that
\begin{flalign*}
\widehat{\omega}_1'(u) &= \frac{2e^{u}\widehat{\phi}_1}{u}\left(1-\frac{1}{u}+\frac{\widehat{\phi}_1'}{\widehat{\phi}_1}-\frac{u}{50\widehat{\phi}_1e^{e^{u}}}\right)
\log{\left(1+\frac{2\lambda}{\left(\sigma-\frac{1}{2}\right)u}\right)} \\
&- \frac{e^{u}\widehat{\phi}_1}{u}
\left(1+\frac{u}{50\widehat{\phi}_1e^{u+e^{u}}}\right)\frac{4\lambda}{u\left(\left(\sigma-\frac{1}{2}\right)u+2\lambda\right)}
+ \frac{\dif{}}{\dif{u}}e^{u}\omega\left(t_0;u\right), \\
\widehat{\omega}_2'(u) &= \frac{\lambda e^{u}}{\pi u}\left(1-\frac{1}{u}\right) + \frac{\dif{}}{\dif{u}}e^{u}\omega\left(t_0;u\right)
\end{flalign*}
with
\begin{multline*}
\frac{\dif{}}{\dif{u}}e^{u}\omega\left(t_0;u\right) = \frac{a_2}{\lambda} + \frac{a_1e^{u}\widehat{\phi}_2}{u\lambda}\left(1-\frac{1}{u}+\frac{\widehat{\phi}_2'}{\widehat{\phi}_2}\right)
- \frac{2a_4\widehat{\phi}_2}{u^2}\left(\frac{1}{u}-\frac{1}{2}\frac{\widehat{\phi}_2'}{\widehat{\phi}_2}\right) \\
-12\lambda e^{2u-e^{u}}\left(\frac{3}{100u^2 e^{2u+e^{u}}}+\frac{3\widehat{\phi}_1}{u^3e^{u}}+\frac{3}{50ue^{u+e^{u}}}+\frac{3\left(1-e^{-u}\right)\widehat{\phi}_1}{2u^2}-\frac{3\widehat{\phi}_1'}{2u^2e^{u}}\right)
\end{multline*}
for $u\in\mathcal{D}$, where
\[
\mathcal{D}\de \left(1,\infty\right)\setminus\left\{\log\log{10^{208}},\log\log{10^{2465}}\right\}.
\]
We are going to demonstrate that $\widehat{\omega}_1'(u)>0$ and $\widehat{\omega}_2'(u)>0$ for
\[
u\in\mathcal{D}_0\de\mathcal{D}\cap\left[\log{\log{\left(3\exp{\left(e^2\right)}\right)}},\infty\right).
\]
Then it will follow that $\widehat{\omega}_1(u)$ and $\widehat{\omega}_2(u)$ increase for all $u\geq \log{\log{\left(3\exp{\left(e^2\right)}\right)}}$ since $\widehat{\omega}_1$ and $\widehat{\omega}_2$ are continuous functions and $\left[\log{\log{\left(3\exp{\left(e^2\right)}\right)}},\infty\right)\setminus\mathcal{D}_0$ is finite.

Because $2\gamma_1\leq t_0\leq 50$, we obtain $18.67\leq\lambda\left(t_0\right)\leq 36.6$, $42.65\leq a_1\left(t_0\right)\leq 44.1$, $0.033\leq a_2\left(t_0\right)\leq 0.044$ and $0.998\leq a_4\left(t_0\right)\leq 1.032$. In order to derive bounds for $\widehat{\phi}_1(u)$, $\widehat{\phi}_2(u)$ and corresponding derivatives, we will divide the set $\mathcal{D}_0$ into three parts, namely:
\begin{enumerate}
  \item The case of $\log{\log{\left(3\exp{\left(e^2\right)}\right)}}\leq u < \log{\log{10^{208}}}$. Here we have
        \[
        \widehat{\phi}_1=0.96, \quad \widehat{\phi}_2=2.491, \quad \widehat{\phi}_1'=\widehat{\phi}_2'=0.
        \]
  \item The case of $\log{\log{10^{208}}}< u < \log{\log{10^{2465}}}$. Here we have
        \[
        \widehat{\phi}_1=0.96, \quad \widehat{\phi}_2 \geq 1.3273, \quad \widehat{\phi}_1'=0, \quad \frac{|\widehat{\phi}_2'|}{\widehat{\phi}_2}\leq 0.32.
        \]
  \item The case of $u>\log{\log{10^{2465}}}$. Here we have
        \[
        0.761\leq \widehat{\phi}_1 \leq 0.96, \quad \widehat{\phi}_2 \geq 0.656, \quad \frac{|\widehat{\phi}_1'|}{\widehat{\phi}_1}\leq 0.1, \quad
        \frac{|\widehat{\phi}_2'|}{\widehat{\phi}_2}\leq 0.2, \quad |\widehat{\phi}_1'|\leq 0.08.
        \]
\end{enumerate}
All these estimates simply follow from definitions for $\widehat{\phi}_1$ and $\widehat{\phi}_2$. It is not hard to see that we can now write
\begin{gather}
\widehat{\omega}_1'(u) \geq \frac{74.68e^{u}\widehat{\phi}_1}{u\left(u+73.2\right)}\left(1-\frac{2}{u}-\frac{|\widehat{\phi}_1'|}{\widehat{\phi}_1}-\frac{u}{25\widehat{\phi}_1 e^{e^{u}}}\right) + \frac{\dif{}}{\dif{u}}e^{u}\omega\left(t_0;u\right), \label{eq:boundderiv} \\
\widehat{\omega}_2'(u) \geq \frac{18.67 e^{u}}{\pi u}\left(1-\frac{1}{u}\right) + \frac{\dif{}}{\dif{u}}e^{u}\omega\left(t_0;u\right) \label{eq:boundderiv3}
\end{gather}
with
\begin{multline}
\label{eq:boundderiv2}
\frac{\dif{}}{\dif{u}}e^{u}\omega\left(t_0;u\right) \geq \frac{e^{u}\widehat{\phi}_2}{u}\left(\frac{42.65}{36.6}\left(1-\frac{1}{u}\right)-\frac{2.064}{u^2e^{u}}-\left(\frac{44.1}{18.67}+\frac{1.032}{ue^{u}}\right)
\frac{|\widehat{\phi}_2'|}{\widehat{\phi}_2}\right) \\
- 439.2e^{2u-e^u}\left(\frac{3}{50ue^{u+e^{u}}}\left(1+\frac{1}{2ue^u}\right)+\frac{3\widehat{\phi}_1}{2u^2}\left(1+\frac{2}{ue^u}\right)+
\frac{3|\widehat{\phi}_1'|}{2u^2e^u}\right),
\end{multline}
where we also used $\log{\left(1+u\right)}\geq u/(1+u)$ in derivation of~\eqref{eq:boundderiv}. By using bounds from each of the above cases (1)--(3) in the inequality~\eqref{eq:boundderiv2}, we obtain the following:
\begin{enumerate}
  \item The case of $\log{\log{\left(3\exp{\left(e^2\right)}\right)}}\leq u < \log{\log{10^{208}}}$. Here we have
        \[
        \frac{\dif{}}{\dif{u}}e^{u}\omega\left(t_0;u\right) \geq 1.413\frac{e^u}{u} - 153.52e^{2u-e^u} > 0.
        \]
  \item The case of $\log{\log{10^{208}}}< u < \log{\log{10^{2465}}}$. Here we have
        \[
        \frac{\dif{}}{\dif{u}}e^{u}\omega\left(t_0;u\right) \geq 0.292\frac{e^u}{u} - 16.62e^{2u-e^u} > 0.
        \]
  \item The case of $u>\log{\log{10^{2465}}}$. Here we have
        \[
        \frac{\dif{}}{\dif{u}}e^{u}\omega\left(t_0;u\right) \geq 0.366\frac{e^u}{u} - 8.47e^{2u-e^u} > 0.
        \]
\end{enumerate}
Because the expressions in the brackets in~\eqref{eq:boundderiv} and~\eqref{eq:boundderiv3} are positive in all of the above cases, it follows $\widehat{\omega}_1'(u)>0$ and $\widehat{\omega}_2'(u)>0$ for all $u\in\mathcal{D}_0$. The proof is thus complete.
\end{proof}

\begin{proof}[Proof of Theorem~\ref{thm:main}.]
The first part of Theorem~\ref{thm:main} follows immediately from Corollary~\ref{cor:main}. Estimates~\eqref{eq:corMainConrete} and~\eqref{eq:corMain2Conrete} follow by taking $t_0=2\gamma_1$ and $u\geq 10$ in Corollary~\ref{cor:main} while observing that then $u\cdot\omega\left(2\gamma_1;u\right)$ is a decreasing function.
\end{proof}

In addition to Corollary~\ref{cor:main}, for our applications it is crucial to estimate $\left|1/\zeta(s)\right|$ and $\left|\zeta(s)\right|$ also for $0\leq t\leq 2\exp{\left(e^2\right)}$. We do this by using numerical methods.

\begin{lemma}
\label{lem:zetabound}
Let $s=\sigma+\ie t$ with $1/2<\sigma\leq 3/2$ and $0\leq t\leq 2\exp{\left(e^2\right)}$. Then
\[
\left|\frac{1}{\zeta(s)}\right| \leq \frac{4}{\sigma-\frac{1}{2}}.
\]
\end{lemma}

\begin{proof}
Let
\[
F_n(s) \de \frac{s-\frac{1}{2}-\ie\gamma_n}{\zeta(s)}
\]
for $n\in\N$, and define sets
\begin{gather*}
\mathcal{S}_n \de \left\{z\in\C\colon \frac{1}{2}\leq \Re\{z\}\leq \frac{3}{2}, \frac{\gamma_{n-1}+\gamma_{n}}{2}\leq \Im\{z\}\leq \frac{\gamma_{n}+\gamma_{n+1}}{2}\right\}, \\
\mathcal{S}_{0} \de \left\{z\in\C\colon \frac{1}{2}\leq \Re\{z\}\leq \frac{3}{2}, 0\leq \Im\{z\}\leq 11\right\},
\end{gather*}
where $\gamma_0\de 22-\gamma_1$. Observe that
\[
\mathcal{S}\de \left\{z\in\C\colon \frac{1}{2}\leq \Re\{z\}\leq \frac{3}{2}, 0\leq \Im\{z\}\leq 2e^{e^{2}}\right\} \subset \mathcal{S}_{0}\cup \bigcup_{n=1}^{2703}\mathcal{S}_n.
\]
Because $F_n(s)$ is a holomorphic function in a neighborhood of $\mathcal{S}_n$ for each $n\in\N$, and $1/\zeta(s)$ is a holomorphic function in a neighborhood of $\mathcal{S}_{0}$, the maximum modulus principle asserts that
\begin{equation}
\label{eq:maxprinc}
\left|\frac{1}{\zeta(s)}\right| \leq \frac{1}{\sigma-\frac{1}{2}} \cdot \max\left\{\max_{1\leq n\leq 2703}\left\{\max_{z\in\partial{\mathcal{S}_n}}\left\{\left|F_n(z)\right|\right\}\right\},
\max_{z\in\partial{\mathcal{S}_0}}\left\{\left|\frac{1}{\zeta(z)}\right|\right\}\right\}
\end{equation}
for $s\in\mathcal{S}$. For each $n\in\{0,1,\ldots,2703\}$ we calculated (using \emph{Mathematica}) values of $\left|F_n(z)\right|$ for $n\neq 0$ and $\left|1/\zeta(z)\right|$ for $n=0$, respectively, at $101$ points for each edge of $\partial{\mathcal{S}_n}$ and points are placed equidistantly to each other. After making further inspection by performing calculations on more points for those $n$ for which the latter procedure returned values greater than $2$, we obtain
\[
\max_{z\in\partial{\mathcal{S}_0}}\left\{\left|\frac{1}{\zeta(z)}\right|\right\} \leq 2, \quad \max_{1\leq n\leq 2703}\left\{\max_{z\in\partial{\mathcal{S}_n}}\left\{\left|F_n(z)\right|\right\}\right\} \leq 3.3,
\]
and the latter expression attains extremum for $n=922$ and $n=923$. The lemma now follows by~\eqref{eq:maxprinc} and after rounding the constants to the largest integers.
\end{proof}

\begin{lemma}
\label{lem:zetabound2}
Let $s=\sigma+\ie t$ with $1/2\leq\sigma\leq 3/4$ and $0\leq t\leq 2\exp{\left(e^2\right)}$. Then $\left|\zeta(s)\right|\leq 14$.
\end{lemma}

\begin{proof}
Let
\[
\mathcal{T} \de \left\{z\in\C\colon \frac{1}{2}\leq \Re\{z\}\leq \frac{3}{4}, 0\leq \Im\{z\}\leq \gamma_{2703}\right\}.
\]
The maximum modulus principle asserts that $\left|\zeta(s)\right|\leq \max_{z\in\partial{\mathcal{T}}}\left|\zeta(z)\right|$ for $s\in\mathcal{T}$. Using \emph{Mathematica} we obtained
\[
\max_{\gamma_1\leq t\leq \gamma_{2703}}\left|\zeta\left(\frac{1}{2}+\ie t\right)\right| \leq 13.5, \quad
\max_{\gamma_1\leq t\leq \gamma_{2703}}\left|\zeta\left(\frac{3}{4}+\ie t\right)\right| \leq 6.91
\]
by calculating values of $\left|\zeta(s)\right|$ at $101$ equidistantly distributed points for each gap between consecutive zeros. Similarly we can verify that $\left|\zeta(s)\right|$ is less than $3.5$ on the rest of $\partial{\mathcal{T}}$. The lemma now follows after rounding the constants to the largest integers.
\end{proof}

\section{Explicit truncated Perron's formula with applications}
\label{sec:Perron}

There exist several variants of the truncated Perron's formula, e.g., the classical version~\cite[Lemma 3.12]{Titchmarsh} which implies~\eqref{eq:PerronNonexplicit} and~\eqref{eq:PerronNonexplicit2}, and the version with a smooth truncation~\cite{RamanaRamare} which produces the error term without the log-factor for suitable chosen test functions. The following theorem is an explicit version of the classical variant from~\cite{RamareEigen} with the error term which is good enough for our purposes. In Section~\ref{sec:general} we derive general estimates for the Mertens function $M(x)$ and the number of $k$-free numbers $Q_{k}(x)$.

\begin{theorem}
\label{thm:truncatedPerron}
Let $f(s)=\sum_{n=1}^{\infty}a_n n^{-s}$ be the Dirichlet series with the abscissa of absolute convergence $\bar{\sigma}$, and let $g(\sigma)=\sum_{n=1}^{\infty}\left|a_n\right|n^{-\sigma}$ for $\sigma>\bar{\sigma}$. If $x\geq 1$, $T\geq 1$, $\sigma>\max\{0,\bar{\sigma}\}$ and $\left|a_n\right|\leq \psi(n)$ for an increasing positive function $\psi$, then
\[
\left|\sum_{n\leq x}a_n - \frac{1}{2\pi\ie}\int_{\sigma-\ie T}^{\sigma+\ie T}f(z)\frac{x^z}{z}\dif{z}\right| \leq 2g(\sigma)\frac{x^{\sigma}}{T} +
4e^{\sigma}\left(\frac{e\psi(ex)x\log{T}}{T} + \psi(ex)\right).
\]
\end{theorem}

\begin{proof}
By~\cite[Theorem 7.1]{RamareEigen} we know that the left-hand side of the latter inequality is not greater than
\[
\frac{2x^{\sigma}}{T}\int_{1/T}^{\infty}\frac{1}{u^2}\sum_{\left|\log{\frac{x}{n}}\right|\leq u}\frac{\left|a_n\right|}{n^{\sigma}}\dif{u} =
\frac{2x^{\sigma}}{T}\left(\int_{1/T}^{1}+\int_{1}^{\infty}\right)\frac{1}{u^2}\sum_{x/e^{u}\leq n\leq xe^{u}}\frac{\left|a_n\right|}{n^{\sigma}}\dif{u}.
\]
We need to bound the last two integrals. Trivially,
\[
\int_{1}^{\infty}\frac{1}{u^2}\sum_{x/e^{u}\leq n\leq xe^{u}}\frac{\left|a_n\right|}{n^{\sigma}}\dif{u} \leq g(\sigma).
\]
Let $\sigma\neq 1$. By partial summation we have
\[
\sum_{a\leq n\leq b}\frac{1}{n^{\sigma}} \leq a^{-\sigma} + \lfloor{b}\rfloor b^{-\sigma} - \lfloor{a}\rfloor a^{-\sigma} + \sigma\int_{a}^{b}\frac{\lfloor{y}\rfloor}{y^{1+\sigma}}\dif{y} \leq 2a^{-\sigma} + (b-a)a^{-\sigma}
\]
for $0<a\leq b$ since $\left|a^{1-\sigma}-b^{1-\sigma}\right|\leq (b-a)\left|\sigma-1\right|a^{-\sigma}$. It is not hard to see that the same inequality holds also for $\sigma=1$ since $\log{(b/a)}\leq b/a-1$. Therefore,
\[
\int_{1/T}^{1}\frac{1}{u^2}\sum_{x/e^{u}\leq n\leq xe^{u}}\frac{\left|a_n\right|}{n^{\sigma}}\dif{u} \leq
2e^{\sigma}\psi(ex)\frac{T}{x^{\sigma}}\left(1+\frac{ex\log{T}}{T}\right),
\]
where we used $e^{u}-e^{-u}\leq 2eu$ which is valid for $u\in[0,1]$. The final estimate from Theorem~\ref{thm:truncatedPerron} now easily follows.
\end{proof}

Let $\sigma>1$. By partial summation one can easily prove that
\begin{equation}
\label{eq:simpleboundzeta}
\zeta(\sigma) \leq \frac{\sigma}{\sigma-1}.
\end{equation}
Although not needed here, note that better estimate exists, see~\cite[Lemma 5.4]{Ramare}. We are going to use~\eqref{eq:simpleboundzeta} in the following corollaries to Theorem~\ref{thm:truncatedPerron}.

\begin{corollary}
\label{cor:Perron}
Let $x\geq e^{2}$. Then
\[
M(x) = \frac{1}{2\pi\ie}\int_{1+\frac{1}{\log{x}}-\ie e\sqrt{x}}^{1+\frac{1}{\log{x}}+\ie e\sqrt{x}}\frac{x^{z}}{z\zeta(z)}\dif{z} + P_1,
\]
where $\left|P_1\right| \leq 22.2\sqrt{x}\log{(ex)}$.
\end{corollary}

\begin{proof}
Taking $a_n=\mu(n)$, $\psi\equiv1$, $\sigma=1+1/\log{x}$ and $T=e\sqrt{x}$ in Theorem~\ref{thm:truncatedPerron}, which conditions are clearly satisfied with such choice of parameters, and then using~\eqref{eq:simpleboundzeta} since $g(\sigma)\leq \zeta(\sigma)$ will give the stated estimate.
\end{proof}

\begin{corollary}
\label{cor:Perron2}
Let $x\geq e^{2}$, $y\geq 1$ and $k\geq 2$ be an integer. Then
\[
Q_{k,2}(x) = \frac{1}{2\pi\ie} \int_{1+\frac{1}{\log{x}}-\ie ex}^{1+\frac{1}{\log{x}}+\ie ex}\zeta(z)f_{y}(kz)\frac{x^z}{z}\dif{z} + P_2,
\]
where $f_y(s)$ is defined by~\eqref{eq:fy}, and $\left|P_2\right|\leq 26y\log{(ex)}$.
\end{corollary}

\begin{proof}
Let $\zeta(s)f_{y}(ks)=\sum_{n=1}^{\infty}a_n n^{-s}$ be the Dirichlet series. By Lemma~\ref{lem:zetafy} we have $\left|a_n\right|\leq y$. Therefore, taking $\psi\equiv y$, $\sigma=1+1/\log{x}$ and $T=ex$ in Theorem~\ref{thm:truncatedPerron} and then using~\eqref{eq:simpleboundzeta} since $g(\sigma)\leq y\zeta(\sigma)$ will furnish the proof.
\end{proof}

Constants from Corollaries~\ref{cor:Perron} and~\ref{cor:Perron2} can be improved in terms of larger $x_0\leq x$, but such an improvement is negligible for our applications.

\subsection{General form of Theorem~\ref{thm:FinalApp}}
\label{sec:general}

We are going to formulate and prove general bounds for $M(x)$ and $Q_{k}(x)$, where constants are expressed with functions developed in previous sections.

\begin{theorem}
\label{thm:MainMoebius}
Assume the Riemann Hypothesis. Let $1/2<\sigma_0<1$, $2\gamma_1\leq t_0\leq 50$ and $x\geq x_0\geq 4\exp{\left(2e^2\right)}$. Then
\[
\left|M(x)\right| \leq \mathcal{N}_1\left(\sigma_0,t_0,x_0\right)x^{\sigma_0+\frac{1}{2}\omega_{0,1}}\log{x},
\]
where
\begin{equation}
\label{eq:omega01}
\omega_{0,1} = \omega_{0,1}\left(\sigma_0,t_0,x_0\right) \de \omega_1\left(\sigma_0,t_0;\log{\log{\left(\frac{3e\sqrt{x_0}}{2}\right)}}\right)
\end{equation}
with $\omega_1\left(\sigma_0,t_0;u\right)$ defined by~\eqref{eq:omega1}, and
\begin{multline}
\label{eq:N1}
\mathcal{N}_1\left(\sigma_0,t_0,x_0\right) \de \frac{1}{\pi}\left(\frac{3e}{2}\right)^{\omega_{0,1}}\left(\frac{1}{2}+\frac{1}{\log{x_0}}+\frac{1.1-\sigma_0}{x_0^{\sigma_0-\frac{1}{2}}\log{x_0}}\right) \\
+ \frac{1}{x_0^{\frac{1}{2}\omega_{0,1}}}\left(\frac{8e^{e^2}}{\pi\sigma_0\left(\sigma_0-\frac{1}{2}\right)\log{x_0}}+
\frac{23.6}{x_0^{\sigma_0-\frac{1}{2}}}\right).
\end{multline}
\end{theorem}

\begin{proof}
Take $T=e\sqrt{x}$ and observe that $T\geq 2\exp{\left(e^2\right)}$. By Cauchy's formula we have
\[
\int_{1+\frac{1}{\log{x}}-\ie T}^{1+\frac{1}{\log{x}}+\ie T}\frac{x^{z}}{z\zeta(z)}\dif{z} = \left(\int_{1+\frac{1}{\log{x}}-\ie T}^{\sigma_0-\ie T}+\int_{\sigma_0-\ie T}^{\sigma_0+\ie T}+\int_{\sigma_0+\ie T}^{1+\frac{1}{\log{x}}+\ie T}\right)\frac{x^{z}}{z\zeta(z)}\dif{z}
\]
since under RH the integrand is a holomorphic function for $\Re\{z\}>1/2$. Denote by $\mathcal{I}_1$, $\mathcal{I}_2$ and $\mathcal{I}_3$ the latter integrals, written in the same order. By Corollary~\ref{cor:Perron} we then have
\begin{equation}
\label{eq:boundMmain}
\left|M(x)\right| \leq \frac{1}{2\pi}\left(\left|\mathcal{I}_1\right| + \left|\mathcal{I}_2\right| + \left|\mathcal{I}_3\right|\right) + 22.2\sqrt{x}\log{(ex)}.
\end{equation}
We need to estimate each of the integrals. Corollary~\ref{cor:main} guarantees that
\begin{equation}
\label{eq:I1I3}
\left|\mathcal{I}_1\right| + \left|\mathcal{I}_3\right| \leq 2\left(1+\frac{1}{\log{x}}-\sigma_0\right)\left(\frac{3e}{2}\right)^{\omega_{0,1}}x^{\frac{1}{2}+\frac{1}{2}\omega_{0,1}}.
\end{equation}
By Corollary~\ref{cor:main} and Lemma~\ref{lem:zetabound} we also have
\begin{equation}
\label{eq:I2}
\left|\mathcal{I}_2\right| \leq \frac{16e^{e^2}}{\sigma_0\left(\sigma_0-\frac{1}{2}\right)}x^{\sigma_0} + 2\left(\frac{1}{2}+\frac{1}{\log{x}}\right)\left(\frac{3e}{2}\right)^{\omega_{0,1}}x^{\sigma_0+\frac{1}{2}\omega_{0,1}}\log{x}.
\end{equation}
Taking~\eqref{eq:I1I3} and~\eqref{eq:I2} into~\eqref{eq:boundMmain} gives the main estimate.
\end{proof}

Before proceeding to the formulation and proof of the similar result also for $Q_{k}(x)$, we need to obtain conditional estimate for $\left|f_y(s)\right|$. We do this entirely with the help of Theorem~\ref{thm:MainMoebius}.

\begin{lemma}
\label{lem:fy}
Assume the Riemann Hypothesis. Let $s=\sigma+\ie t$ with $\sigma>1$, $1/2<\sigma_0<1$, $2\gamma_1\leq t_0\leq 50$, $x\geq x_0\geq 4\exp{\left(2e^2\right)}$, and $y\geq x_0$. If
\begin{equation}
\label{eq:alpha}
\alpha=\alpha\left(\sigma_0,t_0,x_0\right) \de \sigma_0 + \frac{1}{2}\omega_{0,1}\left(\sigma_0,t_0,x_0\right) < 1,
\end{equation}
where $\omega_{0,1}\left(\sigma_0,t_0,x_0\right)$ is defined by~\eqref{eq:omega01}, then
\[
\left|f_{y}(s)\right| \leq \mathcal{N}_1\left(\sigma_0,t_0,x_0\right)\Phi\left(\sigma,\alpha,x_0\right)y^{\alpha-\sigma}\log{y},
\]
where $f_{y}(s)$ is defined by~\eqref{eq:fy} and
\begin{equation}
\label{eq:Phi}
\Phi\left(\sigma,\alpha,x_0\right) \de 1+\frac{\sigma}{\sigma-\alpha}\left(1+\frac{1}{\left(\sigma-\alpha\right)\log{x_0}}\right).
\end{equation}
\end{lemma}

\begin{proof}
By partial summation we have
\[
\sum_{y<n\leq Y} \frac{\mu(n)}{n^s} = \frac{M(Y)}{Y^s} - \frac{M(y)}{y^s} + \sigma\int_{y}^{Y}\frac{M(u)}{u^{1+s}}\dif{u}
\]
for $Y>y$. Taking absolute values in the latter equality and then using Theorem~\ref{thm:MainMoebius} to estimate the right-hand side, we see that we can take $Y\to\infty$ since $\alpha-\sigma<0$. The stated inequality now easily follows.
\end{proof}


\begin{theorem}
\label{thm:MainkFree}
Assume the Riemann Hypothesis. Let $1/2<\sigma_0<1$, $2\gamma_1\leq t_1\leq 50$, $2\gamma_1\leq t_2\leq 50$, $x\geq x_0^{k+1}$, $x_0\geq 4\exp{\left(2e^2\right)}$ and $\alpha<1$, where $\alpha\left(\sigma_0,t_1,x_0\right)$ is defined by~\eqref{eq:alpha}. Then
\[
\left|Q_k(x) - \frac{x}{\zeta(k)}\right| \leq \mathcal{N}_2\left(\sigma_0,t_1,t_2,x_0,k,\alpha\right) x^{\frac{1/2+\alpha}{k+1}+\omega_{0,2}}\log^{2}{x},
\]
where
\[
\omega_{0,2} = \omega_{0,2}\left(t_2,x_0\right) \de \omega_2\left(t_2;\log{\log{\left(\frac{3e x_0}{2}\right)}}\right)
\]
with $\omega_2\left(t_2;u\right)$ defined by~\eqref{eq:omega2}, and
\begin{flalign*}
\mathcal{N}_2\left(\sigma_0,t_1,t_2,x_0,k,\alpha\right) &\de \frac{\mathcal{N}_1\Phi}{\pi(k+1)}\left(\frac{3e}{2}\right)^{\omega_{0,2}}\left(1+\frac{1}{\log{x_0}}
+\frac{0.57}{\sqrt{x_0}\log^{2}{x_0}}\right) \\
&+ \frac{1}{x_0^{\omega_{0,2}}\log{x_0}}
\left(\frac{\mathcal{N}_1}{k+1}\left(\frac{56e^{e^{2}}\Phi}{\pi}+\frac{1}{2x_0^{\frac{1}{2(k+1)}}}\right)
+\frac{27.7}{x_0^{\frac{\alpha-1/2}{k+1}}}\right)
\end{flalign*}
with $\mathcal{N}_1=\mathcal{N}_1\left(\sigma_0,t_1,x_0\right)$ and $\Phi=\Phi\left(k/2,\alpha,x_0\right)$ defined by~\eqref{eq:N1} and~\eqref{eq:Phi}, respectively.
\end{theorem}

\begin{proof}
Take $T=ex$, $y=x^{1/(k+1)}$ and $1/2<\sigma'\leq 3/4$. Observe that $x\geq x_0$, $T\geq 2\exp{\left(e^2\right)}$ and $y\geq x_0$. By Cauchy's formula we have
\begin{multline*}
\int_{1+\frac{1}{\log{x}}-\ie T}^{1+\frac{1}{\log{x}}+\ie T}\zeta(z)f_{y}(kz)\frac{x^z}{z}\dif{z} = \\
\left(
\int_{1+\frac{1}{\log{x}}-\ie T}^{\sigma'-\ie T} + \int_{\sigma'-\ie T}^{\sigma'+\ie T} + \int_{\sigma'+\ie T}^{1+\frac{1}{\log{x}}+\ie T}\right)\zeta(z)f_{y}(kz)\frac{x^z}{z}\dif{z} + \left(2\pi\ie\right)xf_y(k).
\end{multline*}
since under RH the integrand is a holomorphic function for $\Re\{z\}>1/2$ and $z\neq 1$, having a simple pole at $z=1$. Denote by $\mathcal{I}_1$, $\mathcal{I}_2$ and $\mathcal{I}_3$ the latter integrals, written in the same order. By~\eqref{eq:Q1} and Corollary~\ref{cor:Perron2} we then have
\begin{equation*}
Q_k(x) - \frac{x}{\zeta(k)} = \frac{1}{2\pi\ie}\left(\mathcal{I}_1 + \mathcal{I}_2 + \mathcal{I}_3\right) + P_2 - \frac{1}{2}M(y) - S_k(x,y),
\end{equation*}
which by Corollary~\ref{cor:Perron2} and Theorem~\ref{thm:MainMoebius} immediately implies
\begin{multline}
\label{eq:mainQbound}
\left|Q_k(x) - \frac{x}{\zeta(k)}\right| \leq \frac{1}{2\pi}\left(\left|\mathcal{I}_1\right| + \left|\mathcal{I}_2\right| + \left|\mathcal{I}_3\right|\right) \\
+ 26x^{\frac{1}{k+1}}\log{(ex)} + \frac{\mathcal{N}_1}{2(k+1)}x^{\frac{\alpha}{k+1}}\log{x} + \frac{1}{2}x^{\frac{1}{k+1}}.
\end{multline}
Corollary~\ref{cor:main} and Lemma~\ref{lem:fy} guarantee that
\begin{equation}
\label{eq:I11I33}
\left|\mathcal{I}_1\right| + \left|\mathcal{I}_3\right| \leq \frac{2\left(1+\frac{1}{\log{x}}-\sigma'\right)\left(\frac{3e}{2}\right)^{\omega_{0,2}}\mathcal{N}_1}{k+1}
\Phi\left(k\sigma',\alpha,x_0\right)x^{\frac{\alpha-k\sigma'}{k+1}+\omega_{0,2}}
\end{equation}
since $\Phi\left(\sigma,\alpha,x_0\right)y^{\alpha-\sigma}$ is a decreasing function in $\sigma$. By Corollary~\ref{cor:main} and Lemma~\ref{lem:zetabound2} we also have
\begin{multline}
\label{eq:I22}
\left|\mathcal{I}_2\right| \leq \frac{56e^{e^2}\mathcal{N}_1}{\sigma'(k+1)}\Phi\left(k\sigma',\alpha,x_0\right) x^{\frac{\sigma'+\alpha}{k+1}}\log{x} \\
+ \frac{2\left(\frac{3e}{2}\right)^{\omega_{0,2}}\left(1+\frac{1}{\log{x}}\right)\mathcal{N}_1}{k+1}\Phi\left(k\sigma',\alpha,x_0\right) x^{\frac{\sigma'+\alpha}{k+1}+\omega_{0,2}}\log^{2}{x}.
\end{multline}
Taking~\eqref{eq:I11I33} and~\eqref{eq:I22} into~\eqref{eq:mainQbound}, and then letting $\sigma'\to 1/2$, gives the final estimate.
\end{proof}

\section{Proof of Theorem~\ref{thm:FinalApp} and Corollary~\ref{cor:m}}
\label{sec:proofApp}

We are now in the position to prove estimates~\eqref{eq:thmMertensGen}, \eqref{eq:thmkFree}, \eqref{eq:thmMertens} and~\eqref{eq:thmSquarefree} from Theorem~\ref{thm:FinalApp}, and estimate~\eqref{eq:m} from Corollary~\ref{cor:m}. All numerical computations have been made with \emph{Mathematica}.

\begin{proof}[Proof of~\eqref{eq:thmMertensGen}]
We are using Theorem~\ref{thm:MainMoebius}. Take $t_0=38.0820263$,
\[
x_0\geq 10^{10^{4.487}}, \quad u=\log{\log{\left(\frac{3e\sqrt{x_0}}{2}\right)}}, \quad \sigma_0=\frac{1}{2}+\frac{0.842996}{u}.
\]
Our choice for the constants will be clear from the proof~\eqref{eq:thmMertens}, see the first row in Table~\ref{tab:Proof1}. By the definition of the function $\omega_1$, see~\eqref{eq:omega1}, we have
\[
u\cdot\omega_1\left(\sigma_0,t_0;u\right) = 2\left(\widehat{\phi}_1\left(u\right)+\frac{u}{50e^{u+e^{u}}}\right)
\log{\left(1+\frac{2\lambda\left(t_0\right)}{0.842996}\right)} + u\cdot\omega\left(t_0;u\right).
\]
This function is decreasing in $u\geq10$ which implies
\[
\frac{7.3}{u} \leq \omega_{0,1}\left(\sigma_0,t_0,x_0\right) \leq \frac{8.764095}{u}.
\]
With this we can show that $\sigma_0+\frac{1}{2}\omega_{0,1} \leq 1/2 + 5.2251u^{-1}$ and $\mathcal{N}_1\left(\sigma_0,t_0,x_0\right)\leq 0.6$. By taking $x_0=x$, inequality~\eqref{eq:thmMertensGen} now easily follows.
\end{proof}

\begin{proof}[Proof of~\eqref{eq:thmkFree}]
We are using Theorem~\ref{thm:MainkFree}. Take $t_1=30.3977424$, $t_2=2\gamma_1$,
\[
x_0 \geq 10^{10^{23.147}}, u_1=\log{\log{\left(\frac{3e\sqrt{x_0}}{2}\right)}}, u_2=\log{\log{\left(\frac{3e x_0}{2}\right)}}, \sigma_0 = \frac{1}{2} + \frac{0.75782}{u_2}.
\]
Our choice for the constants will be clear from the proof~\eqref{eq:thmSquarefree}, see the first row in Table~\ref{tab:Proof2}. By the definition of the function $\omega_1$, see~\eqref{eq:omega1}, we have
\[
u_1\omega_1\left(\sigma_0,t_1;u_1\right) = 2\left(\widehat{\phi}_1\left(u_1\right)+\frac{u_1}{50e^{u_1+e^{u_1}}}\right)\log{\left(1+\frac{2\lambda\left(t_1\right)u_2}{0.75782u_1}\right)} + u_1\omega\left(t_1;u_1\right).
\]
We have $1<u_2/u_1\leq 1.013$. Similarly as in the previous proof we can deduce from this that
\begin{equation}
\label{eq:2ndproofOmega01}
\frac{7.5}{u_1} \leq \omega_{0,1}\left(\sigma_0,t_1,x_0\right) \leq \frac{7.525}{u_1}.
\end{equation}
By the definition of the function $\omega_2$, see~\eqref{eq:omega2}, we also have
\[
u_2\omega_2\left(t_2;u_2\right) = \frac{1}{\pi}\lambda\left(2\gamma_1\right) + u_2\omega\left(2\gamma_1;u_2\right),
\]
which implies $\omega_{0,2}\left(t_2,x_0\right) \leq 7.492/u_2$. Therefore,
\begin{multline}
\label{eq:2ndproofMain}
\frac{1}{k+1}\left(\frac{1}{2}+\sigma_0+\frac{1}{2}\omega_{0,1}\left(\sigma_0,t_1,x_0\right)\right) + \omega_{0,2}\left(t_2,x_0\right) \\
\leq \frac{1}{k+1} + \frac{7.525}{2(k+1)u_1} + \left(7.492+\frac{0.75782}{k+1}\right)\frac{1}{u_2}.
\end{multline}
In particular, $\alpha=\alpha\left(\sigma_0,t_1,x_0\right)\leq 0.585$. Then this and~\eqref{eq:2ndproofOmega01} imply $\mathcal{N}_1\left(\sigma_0,t_1,x_0\right)\leq 0.2$ and $\Phi\left(k/2,\alpha,x_0\right)\leq 3.41$, which further provides $\mathcal{N}_2\leq 0.1$. The proof is thus furnished by taking $x_0=x^{1/(k+1)}$ in~\eqref{eq:2ndproofMain}.
\end{proof}

\begin{proof}[Proof of~\eqref{eq:thmMertens}]
We are using Theorem~\ref{thm:MainMoebius}. Take $u_0\de \log{\log{\left(3e\sqrt{x_0}/2\right)}}$ and $\log{x_0}=10^{X}\log{10}$. For each $\alpha$ from Table~\ref{tab:FinalApp} we are searching for $\sigma_0\in(1/2,1)$ and $t_0\in\left[30,50\right]$ such that
\begin{equation}
\label{eq:alpha0}
\alpha_0\left(\sigma_0,t_0;u_0\right)\de \sigma_0+\frac{1}{2}\omega_{1}\left(\sigma_0,t_0;u_0\right) < \alpha
\end{equation}
for the smallest possible $u_0$. We do this in the following way: for the particular $u_0$ we find the minimum of $\alpha_0\left(\sigma_0,t_0;u_0\right)$ by calculating this function for each
\[
\sigma_0\in\left\{\frac{1}{2}+\frac{n}{2000} \colon 1\leq n\leq 1000\right\}
\]
while using \texttt{FindMinimum} to determine $t_0$ in each case of $\sigma_0$. With such process we obtain values for $u_0$, $\sigma_0$ and $t_0$, and then we also calculate values for $\mathcal{N}_1$ and $X$, see Table~\ref{tab:Proof1}. Values for $X$ and $A$ from Table~\ref{tab:FinalApp} now simply follow.
\end{proof}

\begin{table}
\centering
\begin{footnotesize}
\begin{tabular}{cccccc}
\toprule
$u_0$ & $\sigma_0$ & $t_0$ & $\alpha_0$ & $\mathcal{N}_1$ & $X$ \\
\midrule
$10.472$ & $0.5805$ & $38.0820263$ & $0.998969...$ & $0.516044...$ & $4.486728...$ \\
$10.600$ & $0.5795$ & $37.7819602$ & $0.989916...$ & $0.504493...$ & $4.542320...$ \\
$12.240$ & $0.5650$ & $34.7754417$ & $0.899710...$ & $0.407781...$ & $5.254575...$ \\
$13.580$ & $0.5575$ & $33.2402484$ & $0.849801...$ & $0.361954...$ & $5.836532...$ \\
$15.460$ & $0.5495$ & $31.9460694$ & $0.799913...$ & $0.321749...$ & $6.653006...$ \\
$18.250$ & $0.5415$ & $31.0325517$ & $0.749867...$ & $0.285883...$ & $7.864688...$ \\
$22.700$ & $0.5330$ & $30.5540678$ & $0.699211...$ & $0.253936...$ & $9.797299...$ \\
$30.080$ & $0.5250$ & $30.4162800$ & $0.649986...$ & $0.226151...$ & $13.00239...$ \\
$45.110$ & $0.5165$ & $30.3958431$ & $0.599987...$ & $0.201251...$ & $19.52983...$ \\
$90.220$ & $0.5085$ & $30.4079746$ & $0.549996...$ & $0.178845...$ & $39.12086...$ \\
\bottomrule
\end{tabular}
\end{footnotesize}
   \caption{Values for the parameters from the proof of~\eqref{eq:thmMertens}.}
   \label{tab:Proof1}
\end{table}

\begin{proof}[Proof of~\eqref{eq:thmSquarefree}]
We are using Theorem~\ref{thm:MainkFree} for $k=2$. Take $u_0=\log{\log{\left(3e x_0/2\right)}}$ and $\log{x_0}=10^{Y}\log{10}$. For each $\beta$ from Table~\ref{tab:FinalApp} we are searching for $\sigma_0\in(1/2,1)$ and $t_1\in\left[30,50\right]$ such that
\[
\beta_0\left(\sigma_0,t_1;u_0\right) \de \frac{1}{3}\left(\frac{1}{2}+\alpha_1\left(\sigma_0,t_1;u_0\right)\right) + \omega_{2}\left(2\gamma_1;u_0\right) < \beta
\]
for the smallest possible $u_0$, where
\[
\alpha_1\left(\sigma_0,t_1;u_0\right)\de \sigma_0+\frac{1}{2}\omega_1\left(\sigma_0,t_1;u_0-0.6932\right).
\]
Note that
\[
u_0-0.6932 \leq u_0+\log{\frac{1+e^{-u_0}\log{\frac{3e}{2}}}{2}} = \log{\log{\left(\frac{3e\sqrt{x_0}}{2}\right)}},
\]
which implies $\alpha\left(\sigma_0,t_1,x_0\right)\leq \alpha_1$. Now the method is the same as in the proof of inequality~\eqref{eq:thmMertens}, and the values for the parameters are listed in Table~\ref{tab:Proof2}. Values for $Y$ and $B$ from Table~\ref{tab:FinalApp} now simply follow.
\end{proof}

\begin{table}
   \centering
\begin{footnotesize}
\begin{tabular}{rcccccc}
\toprule
$u_0$ & $\sigma_0$ & $t_1$ & $\alpha_1$ & $\beta_0$ & $\mathcal{N}_2$ & $Y$ \\
\midrule
$54.13$ & $0.5140$ & $30.3977424$ & $0.584406...$ & $0.499874...$ & $0.085162...$ & $23.14614...$ \\
$54.42$ & $0.5140$ & $30.3999606$ & $0.583951...$ & $0.498984...$ & $0.084900...$ & $23.27209...$ \\
$57.54$ & $0.5130$ & $30.3927510$ & $0.579344...$ & $0.489984...$ & $0.082523...$ & $24.62708...$ \\
$61.45$ & $0.5125$ & $30.4039281$ & $0.574239...$ & $0.479997...$ & $0.079839...$ & $26.32518...$ \\
$65.94$ & $0.5115$ & $30.3989406$ & $0.569128...$ & $0.469992...$ & $0.077357...$ & $28.27516...$ \\
$71.14$ & $0.5105$ & $30.3931181$ & $0.564026...$ & $0.459987...$ & $0.074965...$ & $30.53349...$ \\
$77.23$ & $0.5100$ & $30.4071541$ & $0.558934...$ & $0.449985...$ & $0.072556...$ & $33.17834...$ \\
$84.45$ & $0.5090$ & $30.4008385$ & $0.553851...$ & $0.439997...$ & $0.070337...$ & $36.31395...$ \\
$135.05$ & $0.5055$ & $30.3927192$ & $0.533570...$ & $0.399998...$ & $0.062108...$ & $58.28925...$ \\
$540.10$ & $0.5015$ & $30.4310282$ & $0.508366...$ & $0.349993...$ & $0.053262...$ & $234.2002...$ \\
\bottomrule
\end{tabular}
\end{footnotesize}
   \caption{Values for the parameters from the proof of~\eqref{eq:thmSquarefree}.}
   \label{tab:Proof2}
\end{table}

With the help of Theorem~\ref{thm:MainMoebius} we are able to prove the following simple generalization of Corollary~\ref{cor:m}.

\begin{theorem}
\label{thm:m}
Assume the Riemann Hypothesis. Let $s=\sigma+\ie t$, $1/2<\sigma_0<1$, $2\gamma_1\leq t_0\leq 50$ and $x\geq x_0\geq 4\exp{\left(2e^2\right)}$. Then
\[
\left|\sum_{n\leq x}\frac{\mu(n)}{n^s} - \frac{1}{\zeta(s)}\right| \leq
\mathcal{N}_1\left(1+\frac{|s|}{\sigma-\alpha}\right)\frac{\log{x}}{x^{\sigma-\alpha}} + \frac{\mathcal{N}_1|s|}{\left(\sigma-\alpha\right)^2 x^{\sigma-\alpha}}
\]
for $\sigma>1/2$ and $\alpha<\sigma$, where $\mathcal{N}_1\left(\sigma_0,t_0,x_0\right)$ and $\alpha\left(\sigma_0,t_0,x_0\right)$ are defined by~\eqref{eq:N1} and~\eqref{eq:alpha}, respectively.
\end{theorem}

\begin{proof}
By partial summation and~\eqref{eq:PartSumMobius} we have
\[
\sum_{n\leq x}\frac{\mu(n)}{n^s} = \frac{1}{\zeta(s)} + \frac{M(x)}{x^s} - s\int_{x}^{\infty}\frac{M(u)}{u^{1+s}}\dif{u}.
\]
By Theorem~\ref{thm:MainMoebius} we also have $\left|M(u)\right|\leq \mathcal{N}_1 u^{\alpha}\log{u}$ for $u\geq x$. The final inequality now easily follows after the exact evaluation of the corresponding integral.
\end{proof}

\begin{proof}[Proof of Corollary~\ref{cor:m}]
Take $s=1$, $\sigma_0=0.5795$, $t_0=37.7819602$, and
\[
x_0=\left(\frac{2}{3e}\right)^2 e^{2e^{10.6}}
\]
in Theorem~\ref{thm:m} while using Table~\ref{tab:Proof1} and the fact that $\alpha=\alpha_0$.
\end{proof}


\subsection{Acknowledgements} The author thanks Roger Heath-Brown for having a discussion about Remark~\ref{rem:RHB} and~\eqref{eq:problem}, as well as Richard Brent and Harald Helfgott for useful remarks. Finally, the author is grateful to his supervisor Tim Trudgian for continual guidance and support while writing this manuscript.



\begin{thebibliography}{MOeST21}

\bibitem[BDR08]{Balazard}
M.~Balazard and A.~De~Roton, \emph{Notes de lecture de l'article ``{P}artial
  sums of the {M}\"obius function'' de {K}annan {S}oundararajan}, preprint
  available at arXiv:0810.3587 (2008).

\bibitem[Bou17]{BourgainDecoupling}
J.~Bourgain, \emph{Decoupling, exponential sums and the {R}iemann zeta
  function}, J. Amer. Math. Soc. \textbf{30} (2017), no.~1, 205--224.

\bibitem[BPT21]{BrentPlattTrudgian}
R.~Brent, D.~Platt, and T.~Trudgian, \emph{Accurate estimation of sums over
  zeros of the {R}iemann zeta-function}, Math. Comp. \textbf{90} (2021),
  no.~332, 2923--2935.

\bibitem[Cha18]{Chalker}
K.~A. Chalker, \emph{Perron's formula and resulting explicit bounds on sums},
  M{S}c {T}hesis, Department of Mathematics and Computer Science, University of
  Lethbridge, Lethbridge, 2018.

\bibitem[CDEM07]{CohenDressMarraki}
H.~Cohen, F.~Dress, and M.~El~Marraki, \emph{Explicit estimates for summatory
  functions linked to the {M}\"{o}bius {$\mu$}-function}, Funct. Approx.
  Comment. Math. \textbf{37} (2007), no.~part 1, 51--63.

\bibitem[HS13]{HalupczokSuger}
K.~Halupczok and B.~Suger, \emph{Partial sums of the {M}\"{o}bius function in
  arithmetic progressions assuming {GRH}}, Funct. Approx. Comment. Math.
  \textbf{48} (2013), part 1, 61--90.

\bibitem[Hia16]{HiaryAnExplicit}
G.~A. Hiary, \emph{An explicit van der {C}orput estimate for
  {$\zeta(1/2+it)$}}, Indag. Math. (N.S.) \textbf{27} (2016), no.~2, 524--533.

\bibitem[Ivi03]{Ivic}
A.~Ivi\'{c}, \emph{The {R}iemann {Z}eta-function}, Dover Publications, Inc.,
  Mineola, NY, 2003.

\bibitem[LLS15]{Lamzouri}
Y.~Lamzouri, X.~Li, and K.~Soundararajan, \emph{Conditional bounds for the
  least quadratic non-residue and related problems}, Math. Comp. \textbf{84}
  (2015), no.~295, 2391--2412.

\bibitem[Lan24]{LandauMobius}
E.~Landau, \emph{\"{U}ber die {M}\"{o}biussche {F}unktion}, Rend. di Palermo
  \textbf{48} (1924), 277--280.

\bibitem[LT20]{LanguascoTrudgian}
A.~Languasco and T.~S. Trudgian, \emph{Uniform effective estimates for $|L(1,\chi)|$},
preprint available at arXiv:2011.08348 (2020).

\bibitem[Lit12]{LittlewoodMobius}
J.~E. Littlewood, \emph{Quelques cons\'{e}quences de l'hypoth\`{e}se que la
  fonction $\zeta(s)$ de {R}iemann n'a pas de z\'{e}ros dans le demi-plan
  {$\mathrm{R}(s)>\frac{1}{2}$}}, C. R. Math. Acad. Sci. Paris \textbf{154}
  (1912), 263--266.

\bibitem[MM09]{MaierMontgomery}
H.~Maier and H.~L. Montgomery, \emph{The sum of the {M}\"{o}bius function},
  Bull. Lond. Math. Soc. \textbf{41} (2009), no.~2, 213--226.

\bibitem[MV81]{MontyVaughanSquarefree}
H.~L. Montgomery and R.~C. Vaughan, \emph{The distribution of square-free
  numbers}, Recent Progress in Analytic Number Theory (Durham, 1979), vol.~1,
  London: Academic Press, 1981, pp.~247--256.

\bibitem[MV07]{MontgomeryVaughan}
\bysame, \emph{Multiplicative number theory. {I}. {C}lassical theory},
  Cambridge Studies in Advanced Mathematics, vol.~97, Cambridge University
  Press, Cambridge, 2007.

\bibitem[MOeST21]{MossinghoffOliveiraTrudgian}
M.~J. Mossinghoff, T.~Oliveira~e Silva, and T.~Trudgian, \emph{The distribution
  of $k$-free numbers}, Math. Comp. \textbf{90} (2021), no.~328, 907--929.

\bibitem[Ng04]{Ng2004}
N.~Ng, \emph{The distribution of the summatory function of the {M}\"{o}bius
  function}, Proc. London Math. Soc. (3) \textbf{89} (2004), no.~2, 361--389.

\bibitem[Olv74]{Olver}
F.~W.~J. Olver, \emph{Asymptotics and special functions}, Academic Press, New
  York, 1974.

\bibitem[Pat20]{Patel}
D.~Patel, \emph{An {E}xplicit {U}pper {B}ound for $|\zeta(1+it)|$}, preprint
  available at arXiv:2009.00769 (2020).

\bibitem[Pat21]{PatelPhD}
D.~Patel, \emph{Explicit sub-{W}eyl bound for the {R}iemann zeta function},
PhD {T}hesis, Graduate School, The Ohio State University, 2021.

\bibitem[PT15]{Platt}
D.~J. Platt and T.~S. Trudgian, \emph{An improved explicit bound on
  $\left|\zeta\left(1/2+\textrm{i}t\right)\right|$}, J. Number Theory
  \textbf{147} (2015), 842--851.

\bibitem[RR20]{RamanaRamare}
D.~S. Ramana and O.~Ramar\'{e}, \emph{Variant of the truncated {P}erron formula
  and primes in polynomial sets}, Int. J. Number Theory \textbf{16} (2020),
  no.~2, 309--323.

\bibitem[Ram07]{RamareEigen}
O.~Ramar\'{e}, \emph{Eigenvalues in the large sieve inequality}, Funct. Approx.
  Comment. Math. \textbf{37} (2007), no.~part 2, 399--427.

\bibitem[Ram13]{Ramare2013}
\bysame, \emph{From explicit estimates for primes to explicit estimates for the
  {M}\"{o}bius function}, Acta Arith. \textbf{157} (2013), no.~4, 365--379.

\bibitem[Ram16]{Ramare}
\bysame, \emph{An explicit density estimate for {D}irichlet {$L$}-series},
  Math. Comp. \textbf{85} (2016), no.~297, 325--356.

\bibitem[SS19]{Saha}
B.~Saha and A.~Sankaranarayanan, \emph{On estimates of the {M}ertens function},
  Int. J. Number Theory \textbf{15} (2019), no.~2, 327--337.

\bibitem[Sim21]{SimonicSonRH}
A.~Simoni\v{c}, \emph{On explicit estimates for {$S(t)$}, {$S_1(t)$}, and
  $\zeta(1/2+\mathrm{i}t)$ under the {R}iemann {H}ypothesis}, J. Number Theory
  (2021), https://doi.org/10.1016/j.jnt.2021.05.014.

\bibitem[Sou09]{SoundMobius}
K.~Soundararajan, \emph{Partial sums of the {M}\"{o}bius function}, J. Reine
  Angew. Math. \textbf{631} (2009), 141--152.

\bibitem[Tit27]{TitchConseq}
E.~C. Titchmarsh, \emph{A {C}onsequence of the {R}iemann {H}ypothesis}, J.
  London Math. Soc. \textbf{2} (1927), no.~4, 247--254.

\bibitem[Tit86]{Titchmarsh}
\bysame, \emph{The {T}heory of the {R}iemann {Z}eta-function}, 2nd ed., The
  Clarendon Press, Oxford University Press, New York, 1986.

\bibitem[Tru14]{TrudgianANewUpper}
T.~Trudgian, \emph{A new upper bound for {$|\zeta(1+it)|$}}, Bull. Aust. Math.
  Soc. \textbf{89} (2014), no.~2, 259--264.

\bibitem[Tru15]{TrudgianExplLogDer}
\bysame, \emph{Explicit bounds on the logarithmic derivative and the reciprocal
  of the {R}iemann zeta-function}, Funct. Approx. Comment. Math. \textbf{52}
  (2015), no.~2, 253--261.

\bibitem[Wal63]{WalfiszWeyl}
A.~Walfisz, \emph{Weylsche {E}xponentialsummen in der neuren {Z}ahlentheorie},
  Math. Forschungsberichte 15, VEB Deutscher Verlag der Wissenschaften, Berlin,
  1963.

\end{thebibliography}

\providecommand{\bysame}{\leavevmode\hbox to3em{\hrulefill}\thinspace}
\providecommand{\MR}{\relax\ifhmode\unskip\space\fi MR }
\providecommand{\MRhref}[2]{%
  \href{http://www.ams.org/mathscinet-getitem?mr=#1}{#2}
}
\providecommand{\href}[2]{#2}

\end{document}